\UseRawInputEncoding
\documentclass[12pt, reqno]{amsart}
\usepackage[margin=1in]{geometry}
\usepackage{amssymb,latexsym,amsmath,amscd,amsfonts}
\usepackage{latexsym}
\usepackage[mathscr]{eucal}
\usepackage{bm}
\usepackage{mathptmx}

\def\textmatrix#1&#2\\#3&#4\\{\bigl({#1 \atop #3}\ {#2 \atop #4}\bigr)}
\def\dispmatrix#1&#2\\#3&#4\\{\left({#1 \atop #3}\ {#2 \atop #4}\right)}
\newcommand{\beg}{\begin{equation}}
	\newcommand{\eeg}{\end{equation}}
\newcommand{\ben}{\begin{eqnarray*}}
	\newcommand{\een}{\end{eqnarray*}}

\newtheorem{thm}{Theorem}[section]
\newtheorem{cor}[thm]{Corollary}
\newtheorem{lem}[thm]{Lemma}

\numberwithin{equation}{section} 
\theoremstyle{definition}
\newtheorem{defn}[thm]{Definition}
\newtheorem{rem}[thm]{Remark}
\newtheorem{note}[thm]{Note}

\newcommand{\Wbar}{\underline{W}}
\newcommand{\Vbar}{\underline{V}}
\newcommand{\HS}{\mathcal H}
\newcommand{\C}{\mathbb{C}}
\newcommand{\Ccc}{\mathbb{C}^3}

\newcommand{\D}{\mathbb{D}}
\newcommand{\T}{\mathbb{T}}

\newcommand{\E}{\mathbb E}

\newcommand{\ov}{\overline}

\begin{document}
	\title[Decompositions of contractions]
	{Common reducing subspaces and decompositions of contractions}
	
	\author[Sourav Pal]{Sourav Pal}
	\address[Sourav Pal]{Mathematics Department, Indian Institute of Technology Bombay, Powai, Mumbai - 400076, India.}
	
	\email{souravmaths@gmail.com , sourav@math.iitb.ac.in}

	
	
	\thanks{The author is supported by the Seed Grant of IIT Bombay, the CPDA of the Govt. of India and the MATRICS Award of SERB, (Award No. MTR/2019/001010) of DST, India.}

\begin{abstract}

A commuting triple of Hilbert space operators $(A,B,P)$, for which the closed tetrablock $\overline{\mathbb E}$ is a spectral set, is called a \textit{tetrablock-contraction} or simply an $\mathbb E$-\textit{contraction}, where
\[
\E=\{(x_1,x_2,x_3)\in \C^3:\, 1-x_1z-x_2w+x_3zw \neq 0 \quad \text{ whenever } \; |z|\leq 1, \; \; |w|\leq 1  \} \subset \C^3,
\]
is a polynomially convex domain which is naturally associated with the $\mu$-synthesis problem. By applications of the theory of $\mathbb E$-contractions, we obtain the following main results in this article.\\

\begin{enumerate}

\item It is well known that any number of doubly commuting isometries admit an Wold-type decomposition. Also, examples from the literature show that such decomposition does not hold in general for commuting isometries. We generalize this result to any number of doubly commuting contractions.  Indeed, we present a canonical decomposition to any family of doubly commuting contractions $\mathcal U =\{T_{\lambda}: \lambda \in \Lambda \}$ acting on a Hilbert space $\mathcal H$ and show that $\mathcal H$ continues to decompose orthogonally into joint reducing subspaces of $\mathcal U$ until all members of $\mathcal U$ split together into unitaries and completely non-unitary (c.n.u.) contractions. Naturally, as a special case we achieve an analogous Wold decomposition for any number of doubly commuting isometries.\\

\item Further, we show that a similar canonical decomposition is possible for any family of doubly commuting c.n.u. contractions, where all members jointly orthogonally decompose into pure isometries (i.e. unilateral shifts) and completely non-isometry (c.n.i.) contractions.\\

\item We give an alternative proof to the canonical decomposition of an $\mathbb E$-contraction and apply that to establish independently the following result due to Eschmeier: for any finitely many commuting contractions $T_1, \dots, T_n$ acting on a Hilbert space $\mathcal H$, the space $\mathcal H$ admits an orthogonal decomposition $\mathcal H=\mathcal H_1 \oplus \mathcal H_2$, where $\mathcal H_1$ is the maximal common reducing subspace for $T_1, \dots, T_n$ such that $T_1|_{\mathcal H_1}, \dots , T_n|_{\mathcal H_1}$ are unitaries.

\end{enumerate}

\end{abstract}

\maketitle

\tableofcontents

\section{Introduction}

\vspace{0.2cm}
	
\noindent Throughout the paper, all operators are bounded linear operators acting on complex Hilbert spaces. A contraction is an operator with norm not greater than $1$. We denote by $\C, \D, \T$ the complex plane, the unit disk and the unit circle in the complex plane respectively with center at the origin. \\

\subsection{Motivation.} Finding a common invariant or reducing subspace for a family of commuting operators has always been of independent interest amongst the operator theorists and the first fundamental discovery in this direction was the von Neumann-Wold decomposition of an isometry.

\begin{thm}[\cite{Wold}, von Neumann-Wold]\label{Wold:decomposition}
Let $V$ be an isometry on a Hilbert space $\HS$. Then $\HS$ decomposes into an orthogonal sum $\HS= \HS_0 \oplus \HS_1$ such that $\HS_0$ and $\HS_1$ reduce $V$ and that $V|_{\HS_0}$ is a unitary and $V|_{\HS_1}$ is a pure isometry, i.e., a unilateral shift. This decomposition is uniquely determined; indeed we have
\[
\HS_0=\cap_{n=0}^{\infty} V^n \HS \text{ and } \HS_1=M_+(\mathcal L) \quad \text{ where }\; \; \mathcal L = \HS \ominus V\HS \text{ and } \; M_{+}(\mathcal L)= \oplus_{n=0}^{\infty}V^n\mathcal L.
\]
\end{thm}
Later, the novel idea of determining the maximal unitary part was generalized for any contraction, which in the literature was named the \textit{canonical decomposition} of a contraction.

\begin{thm}[\cite{Nagy}, Chapter-I, Theorem 3.2] \label{thm:CD}
To every contraction $T$ on a Hilbert space $\HS$, there corresponds a decomposition of $\HS$ into an orthogonal sum of two subspaces reducing $T$, say $\HS=\HS_0\oplus \HS_1$ such that $T|_{\HS_0}$ is unitary and $T|_{\HS_1}$ is a completely non-unitary (c.n.u.) contraction; $\HS_0$ or $\HS_1$ may coincide with the trivial subspace $\{0 \}$. This decomposition is uniquely determined. Indeed $\HS_0$ consists of those elements $h\in \HS$ for which
\[
\|T^nh\|=\|h\|=\|T^{*n}h\| \qquad (n=1,2,\dots).
\]
\end{thm}
It is merely mentioned that the orthogonal parts of a contraction are easier to investigate than the original operator. In \cite{N:L}, Levan found a refinement of Theorem \ref{thm:CD} by splitting a c.n.u. contraction further into two orthogonal parts.

\begin{thm}[\cite{N:L}, Theorem 1]\label{thmlv1}
    With respect to a c.n.u. contraction $T$ on $\mathcal{H}$, $\mathcal{H}$ admits the unique orthogonal decomposition
    $\mathcal{H} = \mathcal{H}_1\oplus\mathcal{H}_2$,
    such that $T|_{\mathcal{H}_1}$ is a pure isometry
    and $T|_{\mathcal{H}_2}$ is a completely non-isometry $($c.n.i.$)$ contraction.
\end{thm}

There are several analogues of Wold decomposition for commuting isometries and perhaps the most appealing among them is the one achieved by Berger, Coburn and Lebow in \cite{B-C-Lb}: if $V_1, \dots , V_n$ are commuting isometries on a Hilbert space $\HS$ and if $V=\prod_{i=1}^n V_i$, then the Wold decomposition of $V$ reduces each $V_i$, the tuple $(V_1|_{\HS_1}, \dots , V_n|_{\HS_1})$ consists of unitaries and $(V_1|_{\HS_2}, \dots , V_n|_{\HS_2})$ consists of isometries (whose product is the pure isometry $V|_{\HS_2}$). Later Bercovici, Douglas and Foias generalized this resut by introducing model $n$-isometries, \cite{B-D-F1}. In \cite{M.Slo}, Slocinski established an analogue of Theorem \ref{Wold:decomposition} for a pair of doubly commuting isometries via an independent approach. There the underlying Hilbert space orthogonally split into four parts and the two isometries simultaneously act on them as different combinations of unitaries and unilateral shifts. Slocinski's decomposition was sharp in the sense that no further nontrivial decomposition was possible. It was also shown in \cite{M.Slo, B-K-S} by examples that such decomposition was not possible in general for commuting isometries. Slocinski's result was further generalized in \cite{B-K-S} for any number of doubly commuting isometries with finite dimensional wandering subspaces. Also, we witness canonical decomposition for commuting contractions under certain conditions, see \cite{Al-Pt, Burdak, C-P-S, K:C1, D:P, W.Szy-1} and the references therein. Eschmeier's work \cite{J-Es} contributes remarkably to this program. Indeed, for finitely many commuting contractions $T_1, \dots , T_n$, he could extract the maximal common reducing subspace on which $T_1, \dots , T_n$ act as unitaries.

\subsection{Background material and a brief description of the main results.} 
We establish the results of this article by using heavily the operator theory on the tetrablock. The tetrablock $\E$, where
\begin{equation*}
\E=\{(x_1,x_2,x_3)\in \C^3:\, 1-x_1z-x_2w+x_3zw \neq 0 \quad \text{ whenever } \; |z|\leq 1, \; \; |w|\leq 1  \},
\end{equation*}
 was introduced in \cite{A:W:Y} to study a special case of the famous $\mu$-synthesis problem (see e.g. \cite{Doyle} for a further reading on the $\mu$-synthesis problem). An interested reader is referred to the last section of \cite{A:W:Y}, where there is a clear discussion on the origin of the tetrablock and how it is naturally associated with the interpolation in the $\mu$-unit ball of $\mathcal M_2(\C)$ generated by the $2 \times 2$ diagonal matrices. Indeed, if $E$ is the space of all $2 \times 2$ diagonal complex matrices, then for any complex matrix $A=[a_{ij}]_{2\times 2}$, $\mu_E(A)<1$ if and only if $(a_{11}, a_{22}, \det A) \in \E$.
\begin{defn}
A triple of commuting Hilbert space operators $(A,B,P)$ is called a \textit{tetrablock-contraction} or simply an $\E$-\textit{contraction} if the closed tetrablock $\ov{\E}$ is a spectral set for $(A,B,P)$, that is, the Taylor joint spectrum $\sigma_T(A,B,P) \subseteq \ov{\E}$ and von Neumann inequality
\[
\|f(A,B,P)\|\leq \sup_{z\in\ov{\E}} |f(z)|=\|f\|_{\infty,\; \ov{\E}} 
\]
holds for every complex-valued rational function $f=\dfrac{p}{q}$ with $p,q \in \C[z_1,z_2,z_3]$, such that $q$ does not have any zero inside $\ov{\E}$. Here $f(A,B,P)=p(A,B,P)q(A,B,P)^{-1}$.
\end{defn}

Our first main result is a fine canonical decomposition for any family of doubly commuting contractions. Indeed, in Theorem \ref{thm:DECOMP11}, we show that with respect to any family of doubly commuting contractions $\mathcal U =\{T_{\lambda}: \lambda \in \Lambda \}$ acting on a Hilbert space $\mathcal H$, $\mathcal H$ continues to decompose orthogonally into joint reducing subspaces of $\mathcal U$ until all members of $\mathcal U$ split together into unitaries and completely non-unitary (c.n.u.) contractions. Naturally, as a special case we obtain a von Neumann-Wold decomposition for any number of doubly commuting isometries in Section \ref{sec:6}.

Second, we revisit canonical decomposition of an $\E$-contraction from \cite{S:P-tetra3} and provide an alternative proof to it in Theorem \ref{thm:E-decomp1}. We apply this decomposition to Theorem \ref{thm:decomp1} and find a canonical decomposition for any finitely many commuting contractions, which is an analogue of Theorem \ref{thm:CD}. As a consequence, a Wold-type decomposition for finitely many commuting isometries is achieved in Theorem \ref{thm:Wold1}.

Along with a few preparatory results, the rest of the paper deals with an analogue of Levan's theorem (Theorem \ref{thmlv1}) for any family of doubly commuting c.n.u. contractions. This is presented as Theorem \ref{thm:DECOMP12}.

In Section \ref{sec:2}, we provide a few useful results about operator theory on the tetrablock and also a few definitions and terminologies which will be followed throughout the paper.

Note that we deal with two separate cases, the finite and the infinite ones, when we establish our results for doubly commuting operators. The reason is that some of the proofs in the infinite case use arguments from the finite case. Sincere thanks are due to Professor Hari Bercovici who has made several invaluable comments on this article.

\vspace{0.2cm}

		\section{Definitions, terminologies and preparatory results} \label{sec:2}
	
	\vspace{0.2cm}
	
\noindent Recall that a contraction is an operator whose norm is not greater than $1$. Here we recollect from the literature a few important classes of contractions and introduce some new terminologies for our purposes.

\begin{defn}

A contraction $T$ acting on a Hilbert space $\HS$ is said to be
\begin{itemize}
\item[(i)] a \textit{completely non-unitary contraction} or simply a \textit{c.n.u. contraction} if there is no closed linear subspace of $\HS$ that reduces $T$ and on which $T$ is a unitary ;

\item[(ii)] a \textit{completely non-isometry contraction} or simply a \textit{c.n.i. contraction} if there is no closed linear subspace of $\HS$ that reduces $T$ and on which $T$ is an isometry ;

\item[(iii)] a \textit{pure contraction} or a $C._0$ \textit{contraction} if ${T^*}^nh \rightarrow 0$ as $n \rightarrow \infty $ for every $h\in\HS $ ;

\item[(iv)] a \textit{pure isometry} or a \textit{unilateral shift} if $T$ is an isometry which is also a pure contraction.

\end{itemize}

\end{defn}

\begin{defn}
Let $\mathcal S =\{ T_{\lambda}\,:\, \lambda \in J \}$ be a tuple of commutative Hilbert space contractions. Then $\mathcal S$ is said to be
\begin{itemize}

\item[(i)] \textit{doubly commuting} if $T_{\lambda}T_{\mu}^*=T_{\mu}^*T_{\lambda}$ for each $\lambda , \mu \in J$ with $\lambda \neq \mu$;

\item[(ii)] a \textit{totally unitary tuple} if each $T_{\lambda}$ is a unitary; a \textit{totally isometric tuple} if each $T_{\lambda}$ is an isometry; a \textit{totally c.n.u. tuple} if each $T_{\lambda}$ is a c.n.u. contraction and a \textit{totally c.n.i. tuple} if each $T_{\lambda}$ is a c.n.i. contraction;

\item[(iii)] a \textit{completely non-unitary tuple} or simply a \textit{c.n.u. tuple} if at least one $T_{\lambda}$ is a c.n.u. contraction;

\item[(iv)] a \textit{completely non-isometry tuple} or simply a \textit{c.n.i. tuple} if at least one $T_{\lambda}$ is a c.n.i. contraction;

\item[(v)] a \textit{pure isometry tuple} if all $T_{\lambda}$'s are isometries and at least one of them is a pure isometry, that is, a unilateral shift.

\end{itemize}

\end{defn}	

\vspace{0.2cm}

\subsection{Operator theory on the tetrablock}

\noindent We establish the results of this article by using heavily the operator theory on the tetrablock. The tetrablock $\E$, where
\[
\E=\{(x_1,x_2,x_3)\in \C^3:\, 1-x_1z-x_2w+x_3zw \neq 0 \quad \text{ whenever } \; |z|\leq 1, \; \; |w|\leq 1  \},
\]
 was introduced in \cite{A:W:Y} to study a special case of the famous $\mu$-synthesis problem (see e.g. \cite{Doyle} for a further reading on the $\mu$-synthesis problem). An interested reader is referred to the last section of \cite{A:W:Y}, where there is a clear discussion on the origin of the tetrablock and how it is naturally associated with the interpolation in the $\mu$-unit ball of $\mathcal M_2(\C)$ generated by the $2 \times 2$ diagonal matrices. Indeed, if $E$ is the space of all $2 \times 2$ diagonal complex matrices, then for any complex matrix $A=[a_{ij}]_{2\times 2}$, $\mu_E(A)<1$ if and only if $(a_{11}, a_{22}, \det A) \in \E$. The domain $\E$ is polynomially convex but non-convex (see \cite{A:W:Y}) and has been extensively studied in past two decades both from complex analytic and operator theoretic perspectives (see \cite{A:W:Y, T:B, S:P-tetra1} and references there in). In this section, we recollect from the literature a few operator theoretic results on the tetrablock which will be used in sequel. We begin with a theorem which provides a set of characterizations for the points in the tetrablock and its closure $\ov{\E}$.

\begin{thm}[\cite{A:W:Y}, Theorems 2.2 \& 2.4]\label{thm:21}
Let $(x_1,x_2,x_3)\in \Ccc$. Then the following are equivalent.

\begin{enumerate}

\item $(x_1,x_2,x_3) \in \E \quad (\text{respectively, } \in \overline{\mathbb E})$ ;

\item $|x_1-\bar x_2 x_3|+|x_1x_2-x_3|< 1-|x_2|^2 \;$ (respectively, $\leq 1-|x_2|^2 $ and if $x_1x_2=x_3$ then, in addition $|x_1|\leq 1)$ ;
 
 \item $|x_2-\bar x_1 x_3|+|x_1x_2-x_3|< 1-|x_1|^2 \;$ (respectively, $\leq 1-|x_1|^2 $ and if $x_1x_2=x_3$ then, in addition $|x_2|\leq 1)$ ;
 
 \end{enumerate}	

\end{thm}
The distinguished boundary of the tetrablock was determined in \cite{A:W:Y} to be the following set:
\[
b\ov{\E}=\{(x_1,x_2,x_3)\in \ov{\E}\,:\, |x_3|=1 \}.
\]
Operator theory on the tetrablock was introduced in \cite{T:B} and was further developed in \cite{B-P1} \cite{S:P-tetra2}, \cite{S:P-tetra3} and in a few more articles. Following Arveson's terminology, we consider the commuting operator triples for which $\ov{\E}$ is a spectral set.

\begin{defn}
A triple of commuting Hilbert space operators $(A,B,P)$ is called a \textit{tetrablock-contraction} or simply an $\E$-\textit{contraction} if the closed tetrablock $\ov{\E}$ is a spectral set for $(A,B,P)$, that is, the Taylor joint spectrum $\sigma_T(A,B,P) \subseteq \ov{\E}$ and von Neumann inequality
\[
\|f(A,B,P)\|\leq \sup_{z\in\ov{\E}} |f(z)|=\|f\|_{\infty,\; \ov{\E}} 
\]
holds for every complex-valued rational function $f=\dfrac{p}{q}$ with $p,q \in \C[z_1,z_2,z_3]$, such that $q$ does not have any zero inside $\ov{\E}$. Here $f(A,B,P)=p(A,B,P)q(A,B,P)^{-1}$.
\end{defn}

Unitaries, isometries, co-isometries etc. are special classes of contractions. There are natural analogues of these classes for $\E$-contractions in the literature (see \cite{T:B}) which are given below.

\begin{defn}
Let $A,B,P$ be commuting operators on $\mathcal H$. Then $(A,B,P)$
is called
\begin{itemize}
\item[(i)] an $\mathbb E$-\textit{unitary} if $A,B,P$ are normal
operators and the Taylor joint spectrum $\sigma_T (A,B,P)$ is a
subset of $b\mathbb{E}$ ;

\item[(ii)] an $\mathbb E$-isometry if there exists a Hilbert space
$\mathcal K \supseteq \mathcal H$ and an $\mathbb E$-unitary
$(Q_1,Q_2,V)$ on $\mathcal K$ such that $\mathcal H$ is a joint
invariant subspace of $A,B,P$ and that $(Q_1|_{\mathcal H},
Q_2|_{\mathcal H},V|_{\mathcal H})=(A,B,P)$ ;

\item[(iii)] an $\mathbb E$-co-isometry if the adjoint $(A^*,
B^*,P^*)$ is an $\mathbb E$-isometry.
\end{itemize}

\end{defn}
The following theorem provides a set of independent descriptions of the $\E$-unitaries.

\begin{thm}[{\cite{T:B}, Theorem 5.4}]\label{thm:tu}
    Let $\underline N = (N_1, N_2, N_3)$ be a commuting triple of
    bounded operators. Then the following are equivalent.

    \begin{enumerate}

        \item $\underline N$ is an $\mathbb E$-unitary,

        \item $N_3$ is a unitary and $\underline N$ is an $\mathbb
        E$-contraction,

        \item $N_3$ is a unitary, $N_2$ is a contraction and $N_1 = N_2^*
        N_3$.
    \end{enumerate}
\end{thm}
In the following result, we have a set of characterizations for the $\E$-isometries.

\begin{thm}[{\cite{T:B}, Theorem 5.7}] \label{thm:ti}

    Let $\underline V = (V_1, V_2, V_3)$ be a commuting triple of
    bounded operators. Then the following are equivalent.

    \begin{enumerate}

        \item $\underline V$ is an $\mathbb E$-isometry.

        \item $V_3$ is an isometry and $\underline V$ is an $\mathbb
        E$-contraction.

        \item $V_3$ is an isometry, $V_2$ is a contraction and $V_1=V_2^*
        V_3$.
        
        \item (Wold decomposition) There is a decomposition of $\HS$ into orthogonal direct sum $\HS =\HS_1 \oplus \HS_2$ such that $\HS_1 , \HS_2$ are common reducing subspaces for $V_1,V_2,V_3$ and that $(V_1|_{\HS_1}, V_2|_{\HS_1},V_3|_{\HS_1})$ is an $\E$-unitary and $(V_1|_{\HS_2},V_2|_{\HS_2},V_3|_{\HS_2})$ is a pure $\E$-isometry.

    \end{enumerate}
\end{thm}

It is evident from the theorems above that an $\E$-contraction $(A,B,P)$ is an $\E$-unitary or $\E$-isometry if and only if $P$ is a unitary or isometry respectively. Also canonical decomposition of an $\E$-contraction $(A,B,P)$ acting on a Hilbert space $\HS$ (see Theorem \ref{thm:E-decomp1}) splits $\HS$ into two orthogonal parts $\HS=\HS_1 \oplus \HS_2$ such that $(A|_{\HS_1}, B|_{\HS_1},P|_{\HS_1})$ is an $\E$-unitary and $(A|_{\HS_2},B|_{\HS_2},P|_{\HS_2})$ is an $\E$-contraction with $P|_{\HS_2}$ being a c.n.u. contraction. These facts lead to defining the following natural classes of $\E$-contractions.

\begin{defn}\label{def3}
    Let $(A,B, P)$ be an $\mathbb E$-contraction on a Hilbert space $\mathcal
    H$. We say that $(A,B,P)$ is
    \begin{itemize}
        \item[(i)] a \textit{completely non-unitary} $\E$-\textit{contraction}, or simply a \textit{c.n.u.} $\mathbb E$-\textit{contraction} if $P$ is a
        c.n.u. contraction ;
        \item[(ii)] a \textit{completely non-isometry} $\E$-\textit{contraction}, or simply a \textit{c.n.i.} $\mathbb E$-\textit{contraction} if $P$ is a
        c.n.i. contraction ;
        \item[(iii)] a \textit{pure} $\E$-\textit{contraction} if $P$ is a pure contraction ;
        \item[(iv)] a \textit{pure} $\E$-\textit{isometry} if $P$ is a pure isometry.       
        \end{itemize}

\end{defn}

The next theorem plays central role in the operator theory of the tetrablock. An interested reader is referred to \cite{T:B, B-P1, S:P-tetra1, S:P-tetra2, S:P-tetra3} etc. to witness how this result plays pivotal role in deciphering the structure of an $\E$-contraction and in all kind of Nagy-Foias type operator theory for $\E$-contractions.

\begin{thm}[\cite{T:B}, Theorem 1.3] \label{exist-tetra}
To every $\mathbb
E$-contraction $(A, B, P)$ there were two unique operators $F_1$
and $F_2$ on $\mathcal{\mathcal{D}_P} =\overline{\text{Ran}}(I -
P^*P)$ that satisfied the fundamental equations, i.e,
\[
A-B^*P = D_PF_1D_P\,, \qquad B-A^*P = D_PF_2D_P.
\]
\end{thm}
The operators $F_1,F_2$ are called the \textit{fundamental
operators} of $(A, B, P)$. Taking the fundamental operators as the key ingredient, the following $\E$-isometric dilation was constructed in \cite{T:B} for a certain class of $\E$-contractions.

\begin{thm}[\cite{T:B}, Theorem 6.1]\label{thm:tetra-dilation}

Let $(A, B, P)$ be a tetrablock contraction on
$\mathcal H$ with fundamental operators $F_1$ and $F_2$ . Let $\mathcal D_P$ be the closure of the range of $D_P$.
Let $\mathcal K = \mathcal H \oplus \mathcal D_P \oplus \mathcal D_P
\oplus \cdots = \mathcal H \oplus l^2(\mathcal D_P) $. Consider the
operators $V_1, V_2$ and $V_3$ defined on $\mathcal{K}$ by
\begin{align*} &
V_1(h_0,h_1,h_2,\dots)=(Ah_0,F_2^* D_P h_0 + F_1 h_1 , F_2^*h_1 + F_1 h_2 , F_2^*h_2 + F_1 h_3,\dots)\\
& V_2(h_0,h_1,h_2,\dots)=(Bh_0 , F_1^* D_P h_0 + F_2 h_1 , F_1^*h_1 + F_2 h_2 , F_1^*h_2 + F_2 h_3,\dots)\\
& V_3(h_0,h_1,h_2,\dots)=(Ph_0, D_P h_0,h_1,h_2,\dots).
\end{align*}
Then \begin{enumerate}
\item $\Vbar = (V_1,V_2,V_3)$ is a minimal tetrablock isometric
    dilation of $(A, B, P)$ if $[F_1 , F_2] = 0$ and $[F_1 , F_1^* ] = [F_2 , F_2^* ]$.
\item If there is a tetrablock isometric dilation $\Wbar =
    (W_1,W_2,W_3)$ of $(A, B, P)$ such that $W_3$ is the minimal isometric dilation of $P$,
    then $\Wbar$ is unitarily equivalent to $\Vbar$. Moreover, $[F_1, F_2] = 0$ and $[F_1 , F_1^* ]
    = [F_2 , F_2^* ]$.
\end{enumerate}

\end{thm}

\vspace{0.1cm}

\section{Canonical decomposition of commuting contractions} \label{sec:3}

\vspace{0.4cm}

\noindent  In this Section, we present a canonical decomposition for any finite number of commuting contractions by an application of the canonical decomposition of an $\E$-contraction, which is given below. We present a few preparatory results before going to the main result, viz. Theorem \ref{thm:decomp1}. We begin with a lemma whose proof is a routine exercise.

\begin{lem}\label{lem:23}

If $X\subseteq \mathbb C^n$ is a polynomially convex set, then $X$
is a spectral set for a commuting tuple $(T_1,\dots,T_n)$ if and
only if

\begin{equation}\label{pT}
\|f(T_1,\dots,T_n)\|\leq \| f \|_{\infty,\, X}\,
\end{equation}
for all holomorphic polynomials $f$ in $n$-variables.

\end{lem}

\begin{lem}\label{lem:triangular1}
If $P,Q$ are commuting contractions, then $(P,Q,PQ)$ is an $\E$-contraction.
\end{lem}

\begin{proof}
It follows from part-$(2)$ or part-$(3)$ of Theorem \ref{thm:21} that for any points $x_1,x_2 \in \ov{\D}$, the point $(x_1,x_2,x_1x_2)\in \ov{\E}$. So the map $\phi(x_1,x_2)=(x_1,x_2,x_1x_2)$ embeds $\ov{\D^2}$ analytically into $\ov{\E}$. We have by Ando's inequality that $\ov{\D^2}$ is a spectral set for $(P,Q)$, that is,
\[
\|f(P,Q)\|\leq \|f\|_{\infty, \, \ov{\D}}\,, \quad \; \text{for any } f\in \C[z_1,z_2].
\]
Thus, for any polynomial $g\in \C[z_1,z_2,z_3]$, we have that
\[
\| g(P, Q,PQ) \| =\|  g\circ \phi (P,Q)  \|  \leq \| g\circ \phi \|_{\infty, \ov{\D^2}} = \| g\|_{\infty \,, \phi(\ov{\D^2})}\leq \| g\|_{\infty\,, \ov{\E}}.
\]
Since $\ov{\E}$ is polynomially convex it follows from Lemma \ref{lem:23} that $(P,Q,PQ)$ is an $\E$-contraction.

\end{proof}

\begin{lem}\label{lem:E-1}
Let $A,B$ be commuting contractions on a Hilbert space $\HS$. Then $AB$ is unitary if and only if $A$ and $B$ are unitaries.
\end{lem}

\begin{proof}

If $A,B$ are unitaries, then evidently $AB$ is a unitary. For the converse suppose $AB$ is a unitary. By Spectral Mapping Theorem
\[
\sigma(AB)=\{\lambda_1\lambda_2\,:\, (\lambda_1,\lambda_2)\in \sigma_T(A,B)\}.
\]
Therefore, for any $(a,b)\in \sigma_T(A,B)$, we have that $ab\in\sigma(AB)$ and thus $|ab|=1$. Again since $(A,B)$ is a pair of commuting contractions, it follows that $|a|, |b|\leq 1$. Thus, $|a|=|b|=1$, that is, $(a,b)\in\mathbb T^2$. Again $(A,B,AB)$ is an $\mathbb E$-contraction by Lemma \ref{lem:triangular1} and $AB$ is unitary. Therefore, by Theorem \ref{thm:tu}, $(A,B,AB)$ is an $\mathbb E$-unitary. Thus, by definition $(A,B,P)$ is a triple of commuting normal operators. So, $(A,B)$ is a pair of commuting normal operators whose Taylor joint spectrum lies in $\T^2$. Therefore, $A,B$ are unitaries and the proof is complete.

\end{proof}

\begin{cor}
For commuting contractions $P_1,\dots ,P_n$, $\prod_{i=1}^n P_i$ is unitary if and only if $P_1,\dots ,P_n$ are unitaries.
\end{cor}

\begin{cor}
For an $\E$-unitary $(A,B,P)$, if $P=AB$ then $A,B$ are unitaries.
\end{cor}

\subsection{Canonical decomposition of an $\E$-contraction}

\noindent It was proved in \cite{S:P-tetra3} by the author of this article that every $\E$-contraction admits a canonical decomposition which is analogous to the canonical decomposition of a contraction (Theorem \ref{thm:CD}). Indeed, in Theorem 3.1 of \cite{S:P-tetra3} it was shown by using the positivity of certain operator pencils that an $\E$-contraction can be orthogonally decomposed into two parts of which one is an $\E$-unitary and the other is a c.n.u $\E$-contraction. Applying similar techniques as in the proof of Theorem 5.1 of \cite{B-P1}, we present below an independent and shorter proof of the same result using the existence of the fundamental operators of an $\E$-contraction.

\begin{thm}\label{thm:E-decomp1}
Let $(A,B,P$ be an $\E$-contraction acting on a
Hilbert space $\mathcal H$. Let $\mathcal H_1$ be the maximal
subspace of $\mathcal H$ which reduces $P$ and on which $P$ is
unitary. Let $\mathcal H_2=\mathcal H\ominus \mathcal H_1$. Then
\begin{enumerate}
\item $\mathcal H_1,\mathcal H_2$ reduce $A,B$,

\item $(A|_{\mathcal H_1},B|_{\mathcal H_1},P|_{\mathcal H_1})$ is an $\E$-unitary,

\item $(A|_{\mathcal H_2},B|_{\mathcal H_2},P|_{\mathcal H_2})$ is a completely non-unitary $\E$-contraction.
\end{enumerate}
The subspace $\mathcal H_1$ or $\mathcal H_2$ may be the
trivial subspace $\{0\}$.
\end{thm}

\begin{proof}

    If $P$ is a c.n.u contraction
    then $\mathcal H_1=\{0\}$ and if $P$ is a unitary then $\mathcal H_2=\{0\}$. In such cases the
    theorem is trivial. Suppose $P$ is neither a unitary nor a c.n.u contraction. With respect to the decomposition $\mathcal H=\mathcal H_1\oplus
    \mathcal H_2$, let
    \[
    A=
    \begin{bmatrix}
    A_{11}&A_{12}\\
    A_{21}&A_{22}
    \end{bmatrix}\,,\,
    B=
    \begin{bmatrix}
    B_{11}&B_{12}\\
    B_{21}&B_{22}
    \end{bmatrix}
    \text{ and } P=
    \begin{bmatrix}
    P_1&0\\
    0&P_2
    \end{bmatrix}
    \]
    so that $P_1$ is a unitary and $P_2$ is completely
    non-unitary. Since $P_2$ is completely non-unitary it follows that
    if $x\in\mathcal H$ and
    \[
    \|P_2^nx\|=\|x\|=\|{P_2^*}^nx\|, \quad n=1,2,\hdots
    \]
    then $x=0$. Since $(A, B, P)$ is an $\mathbb E$-contraction, by Theorem \ref{exist-tetra}, there are two unique operators $F_1$ and $F_2$ on $\mathcal{D}_P$ such that,
    \[
    A-B^*P = D_PF_1D_P, \qquad B-A^*P = D_PF_2D_P.
    \]
    With respect to the decomposition $\mathcal{D}_P = \mathcal{D}_{P_1} \oplus \mathcal{D}_{P_2} = \{0\}\oplus\mathcal{D}_{P_2}$, let
    \[ F_i=
    \begin{bmatrix}
    0 & 0\\
    0 & F_{i22}
    \end{bmatrix}
    \text{ for } i = 1, 2.
    \]
    Then from $A-B^*P = D_PF_1D_P$, we have
    \begin{align}
    A_{11}&=B_{11}^*P_1      & A_{12}&=B_{21}^*P_2\,, \label{eqn:1} \\
    A_{21}&=B_{12}^*P_1    & A_{22} - B_{22}^*P_2&=D_{P_2}F_{122}D_{P_2}\,. \label{eqn:2}
    \end{align}
    Similarly from $B-A^*P = D_PF_2D_P$, we have
    \begin{align}
    B_{11}&=A_{11}^*P_1      & B_{12}&=A_{21}^*P_2\,, \label{eqn:3} \\
    B_{21}&=A_{12}^*P_1    & B_{22} - A_{22}^*P_2&=D_{P_2}F_{222}D_{P_2}\,. \label{eqn:4}
    \end{align}

    Since $AP = PA$, so we have
    \begin{align}
    A_{11}P_1&=P_1A_{11}    & A_{12}P_2=P_1A_{12}\,, \label{eqn:5} \\
    A_{21}P_1&=P_2A_{21}    & A_{22}P_2=P_2A_{22}\,. \label{eqn:6}
    \end{align}
    Also by commutativity of $B$ and $P$ we have
    \begin{align}
    B_{11}P_1&=P_1B_{11}    & B_{12}P_2=P_1B_{12}\,, \label{eqn:7} \\
    B_{21}P_1&=P_2B_{21}    & B_{22}P_2=P_2B_{22}\,. \label{eqn:8}
    \end{align}
    Now from the first equation in (\ref{eqn:2}) we have
    \[
    B_{12}^*P_1^2  = A_{21}P_1 = P_2A_{21} = P_2B_{12}^*P_1.
    \]
    Since $P_1$ is a unitary, so $B_{12}^*P_1 = P_2B_{12}^*$. Now
    \begin{equation}\label{eqn:9}
     P_2P_2^*A_{21}  = P_2B_{12}^* = B_{12}^*P_1 = A_{21}
    \end{equation}    
    and
    \begin{equation}\label{eqn:10}
    P_2^*P_2A_{21}  = P_2^*A_{21}P_1 = B_{12}^*P_1 = A_{21}.
    \end{equation}
    From the first equation in (\ref{eqn:6}) we have range of $A_{21}$ is invariant under $P_2$. Again from $A_{21}P_1 = P_2A_{21}$ we have
    \[
    P_2^*A_{21}P_1 = P_2^*P_2A_{21} = A_{21}.
    \]
    Since $P_1$ is unitary, the range of $A_{21}$ is invariant under $P_2^*$. Thus the range of $A_{21}$ is a reducing subspace of $P_2$. Again by (\ref{eqn:9}) and (\ref{eqn:10}) we have that $P_2$ is unitary on the range of $A_{21}$. Since $P_2$ is a c.n.u contraction, we must have $A_{21}=0$. Similarly we can prove that $B_{21} = 0$. Now from (\ref{eqn:1}), $A_{12} = 0$. Thus with respect to the decomposition $\mathcal{H} =\mathcal{H}_1 \oplus \mathcal{H}_2$
    \[
    A=
    \begin{bmatrix}
    A_{11}&0\\
    0&A_{22}
    \end{bmatrix}\,,\,
    B=
    \begin{bmatrix}
    B_{11}&0\\
    0&B_{22}
    \end{bmatrix}.
    \]
So, $\mathcal{H}_1$ and $\mathcal{H}_2$ reduce $A$ and $B$. Now $(A_{22}, B_{22},P_2)$ is the restriction of the $\mathbb E$-contraction $(A, B, P)$ to the reducing subspace $\mathcal{H}_2$. Therefore, $(A_{22}, B_{22}, P_2)$ is an $\mathbb E$-contraction. Since $P_2$ is c.n.u contraction, $(A_{22}, B_{22}, P_2)$ is a c.n.u $\mathbb E$-contraction. Also $(A_{11},B_{11},P_1)$ is an $\mathbb E$-contraction with $P_1$ being a unitary. Therefore, by
Theorem \ref{thm:tu}, $(A_{11},B_{11},P_1)$ is an $\mathbb E$-unitary and the proof is complete.

\end{proof}

Being armed with all these machineries we are now in a position to present the main theorem of this section and this is a generalization of Theorem \ref{thm:CD} for finitely many commuting contractions. Note that this result is a variant of Lemma 2.2 of \cite{J-Es}, which we prove here in an alternative way.

	\begin{thm}\label{thm:decomp1}
	
Let $T_1,\dots, T_n$ be commuting contractions acting on a Hilbert space $\HS$. Then there corresponds a decomposition of $\HS$ into an orthogonal sum of two subspace $\HS_1,\HS_2$ of $\HS$ such that
\begin{enumerate}

\item $\HS_i$ is a joint reducing subspace of $T_1,\dots, T_{n},$ for $i=1,2$;
 
\item $\HS_1$ is the maximal joint reducing subspace of $T_1,\dots, T_n$ such that $T_j|_{\HS_1}$ is a unitary for $j=1,\dots ,n$, that is, $(T_1|_{\HS_1},\dots,T_n|_{\HS_1})$ is a totally unitary tuple;

\item $(T_1|_{\HS_2},\dots,T_n|_{\HS_2})$ is a completely non-unitary tuple;

\item if $T=\prod_{i=1}^n T_i \,,$ then $\HS_1$ is the maximal reducing subspace of $T$ such that $T|_{\HS_1}$ is unitary;

\item $\HS_1$ consists of those elements $h\in \HS$ for which
\[
\|T^nh\|=\|h\|=\|T^{*n}h\|\,, \quad \qquad (n=1,2,\dots)\, ;
\]

\end{enumerate}
This decomposition is uniquely determined and anyone of $\HS_1, \HS_2$ may coincide with the trivial subspace $\{0\}$.

\end{thm}

\begin{proof}

Let $T_i'=\prod T_j$, where $j$ runs over the set $\{1,\dots, i-1,i+1,\dots, n\}$ and also let $T=\prod_{j=1}^n T_j$. By Lemma \ref{lem:triangular1}, $(T_i,T_i',T)$ is an $\E$-contraction for every $i=1,\dots ,n$. By Theorem \ref{thm:E-decomp1}, every $\E$-contraction admits a canonical decomposition. So, if $\HS_1$ is the maximal subspace of $\HS$ that reduces $T$ and on which $T$ is unitary and if $\HS_2=\HS \ominus \HS_1$, then for the $\E$-contraction $(T_i,T_i',T)$ we have by Theorem \ref{thm:E-decomp1} that
\begin{enumerate}
\item $\mathcal H_1,\mathcal H_2$ reduce $T_i,T_i'$,

\item $(T_i|_{\mathcal H_1},T_i'|_{\mathcal H_1},T|_{\mathcal H_1})$ is an $\E$-unitary,

\item $(T_i|_{\mathcal H_2},T_i'|_{\mathcal H_2},T|_{\mathcal H_2})$ is a completely non-unitary $\E$-contraction.
\end{enumerate}
The subspace $\mathcal H_1$ or $\mathcal H_2$ may coincide with the
trivial subspace $\{0\}$. Since $(T_i|_{\mathcal H_1})(T_i'|_{\mathcal H_1})=T|_{\mathcal H_1}$ and since $T|_{\mathcal H_1}$ is unitary, it follows from Lemma \ref{lem:E-1} that $T_i|_{\HS_1}$ is unitary for $i=1,\dots, n$. If $\{0\} \neq\HS_3\subseteq \HS_2$ reduces $T_1,\dots, T_n$ and if $T_i|_{\HS_3}$ is unitary for $i=1,\dots, n$, then $\HS_3$ reduces $T=\prod_{i=1}^n T_i$ and $T|_{\HS_3}$ is unitary which contradicts the fact that $\HS_1$ is the maximal reducing subspace of $T$ such that $T|_{\HS_1}$ is unitary. Hence $\HS_1$ is the maximal subspace of $\HS$ which reduces all of $T_1,\dots ,T_n$ and on which each $T_i$ is unitary. Therefore, at least one of ${T_1|_{\HS_2}}, \dots, T_n|_{\HS_2}$ is a c.n.u. contraction. Consequently we have that$(T_1|_{\HS_1},\dots,T_n|_{\HS_1})$ is a totally unitary tuple and that $(T_1|_{\HS_2},\dots,T_n|_{\HS_2})$ is a c.n.u. tuple. Parts $(4) \,\&\, (5)$ of the theorem follow from the canonical decomposition of the contraction $T$ (see Theorem \ref{thm:CD}). Hence the proof is complete.

\end{proof}

\subsection{Wold decomposition of commuting isometries}
As a special case of Theorem \ref{thm:decomp1}, we obtain the following Wold decomposition of a finite number of commuting isometries.

\begin{thm}\label{thm:Wold1}
	
Let $V_1,\dots, V_n$ be commuting isometries acting on a Hilbert space $\HS$. Then there corresponds a decomposition of $\HS$ into an orthogonal sum of two subspace $\HS_1,\HS_2$ of $\HS$ such that
\begin{enumerate}

\item $\HS_i$ is a joint reducing subspace of $V_1,\dots, V_{n},$ for $i=1,2$;
 
\item $\HS_1$ is the maximal joint reducing subspace of $V_1,\dots, V_n$ such that $V_j|_{\HS_1}$ is a unitary for $j=1,\dots ,n$, that is, $(V_1|_{\HS_1},\dots,V_n|_{\HS_1})$ is a unitary tuple;

\item $(V_1|_{\HS_2},\dots,V_n|_{\HS_2})$ is a pure isometric tuple;

\item if $V=\prod_{i=1}^n V_i \,,$ then
\[
\HS_1=\cap_{n=0}^{\infty} V^n \HS \; \text{ and } \; \HS_2=M_{+}(\mathcal L), \; \text{ where } \; \mathcal L=\HS \ominus V\HS \; \text{ and } \; M_{+}(\mathcal L)= \oplus_{n=0}^{\infty}V^n\mathcal L.
\]
\end{enumerate}
This decomposition is uniquely determined and anyone of $\HS_1, \HS_2$ may be equal to the trivial subspace $\{0\}$.

\end{thm}

\begin{proof}
We follow the same method as of Theorem \ref{thm:decomp1}, where the canonical decomposition of the isometry $V=\prod_{i=1}^n V_i$ is same as the Wold decomposition of $V$. Clearly $(V_i,V_i',V)$ is an $\E$-isometry by part-(2) of Theorem \ref{thm:ti}. In Theorem \ref{thm:decomp1}, we reached the desired conclusion by applying the canonical decomposition of an $\E$-contraction (see Theorem \ref{thm:E-decomp1}). Here, we shall apply the special case, i.e., the Wold decomposition of an $\E$-isometry (see part-(4) of Theorem \ref{thm:ti}). Part-(4) of the theorem follows from the Wold decomposition of the isometry $V$. 

\end{proof}

\section{Canonical decomposition of finitely many doubly commuting contractions} \label{sec:4}

\vspace{0.2cm}

\noindent In this Section, we shall make a refinement of the canonical decomposition described in Theorem \ref{thm:decomp1} for any finite number of doubly commuting contractions. Also, as we have mentioned before that examples exist in the literature (viz. \cite{M.Slo, B-K-S}) which guarantee that this refined decomposition is not valid in general for commuting contractions, in fact it is not valid for commuting isometries. Once again by an application of the canonical decomposition of a couple of certain $\E$-contractions, we achieve our result. As a consequence, we obtain as a special case a refined Wold decomposition for a finite family of doubly commuting isometries. We begin with a few preparatory results.

\begin{thm}\label{thm:E-contraction}
Let $(A, B, P)$ be a commuting triple of contractions acting on a Hilbert space $\HS$. Suppose there exist operators $F_1,F_2\in\mathcal B ({\mathcal D_P})$ that satisfy the following :
\begin{enumerate}
\item $\Omega (F_1+zF_2)\leq 1 \quad$ for all $z\in \ov{\D}$ ;
\item $ A-B^*P=D_PF_1D_P$ and $B-A^*P=D_PF_2D_P\,$ ; 
\item $[F_1,F_2]=0$ and $[F_1^*,F_1]=[F_2^*,F_2]$ .
\end{enumerate}
Then $(A,B,P)$ is an $\E$-contraction that dilates to an $\E$-isometry.
\end{thm}

\begin{proof}

If $A,B,P$ and $F_1,F_2$ satisfy the hypotheses, then we can construct $(V_1,V_2,V_3)$ as in Theorem \ref{thm:tetra-dilation} which by the proof of of Theorem \ref{thm:tetra-dilation} is an $\E$-isometry. Indeed, $(V_1,V_2,V_3)$ satisfies condition-(3) of Theorem \ref{thm:ti} then. Thus $(V_1,V_2,V_3)$ is an $\E$-isometric dilation of $(A,B,P)$. Since every $\E$-isometry, by definition, extends to an $\E$-unitary, it follows that $(A,B,P)$ has an $\E$-unitary dilation. Therefore, according to Arveson's terminology, $\ov{\E}$ is a complete spectral set for $(A,B,P)$ which implies that $\ov{\E}$ is a spectral set for $(A,B,P)$. Hence $(A,B,P)$ is an $\E$-contraction. Again Theorem \ref{thm:ti} clearly shows that $(A,B,P)$ dilates to the $\E$-isometry $(V_1,V_2,V_3)$.

\end{proof}

The following theorem plays the central role in the rest of the paper. A similar result is there in the literature (e.g. see Theorem I.3.2 in \cite{Nagy-Foias}). Our proof here is independent and is based on operator theory on the tetrablock. 

\begin{thm}\label{thm:key}
For a pair of doubly commuting contractions $P,Q$ acting on a Hilbert space $\HS$, if $Q=Q_1\oplus Q_2$ is the canonical decomposition of $Q$ with respect to the orthogonal decomposition $\HS=\HS_1 \oplus \HS_2$, then $\HS_1,\HS_2$ are reducing subspaces for $P$.
\end{thm}

\begin{proof}

Consider the operator triples $$\left( \dfrac{P+P^*Q}{2}, \dfrac{P+P^*Q}{2}, Q \right)\,,\;\left( \dfrac{P-P^*Q}{2}, \dfrac{P-P^*Q}{2}, -Q \right)$$ acting on $\HS$. We first show that they are $\E$-contractions. They are commuting triples of contractions as $P,Q$ are doubly commuting contractions. We shall apply Theorem \ref{thm:E-contraction} to show that they are $\E$-contractions. By the double commutativity of $P,Q$ we have that
\[
PD_Q^2=P(I-Q^*Q)=(I-Q^*Q)P=D_Q^2P
\]
and hence it follows by iteration that
\begin{equation}\label{eqn:a01}
P(D_Q^2)^n=(D_Q^2)^nP \quad \text{ for } n=0,1,2,\dots \;.
\end{equation}
Consequently $Pf(D_Q^2)=f(D_Q^2)P$, for any polynomial $f\in \C [z]$. Choose a sequence of polynomials $\{f_m\}$ in $\C [z]$ that converges to the function $z^{1/2}$ uniformly on the interval $[0,1]$ as $m \rightarrow \infty$. The sequence of operators $\{ f_m(T) \}$ then converges in norm to $T^{1/2}$ for any positive operator $T$, where $T^{1/2}$ denotes the unique positive square root of $T$. For this sequence of polynomials, we obtain from (\ref{eqn:a01}) by taking limit as $m \rightarrow \infty$ that $PD_Q=D_QP$. Now
\[
\dfrac{P+P^*Q}{2}- \left( \dfrac{P+P^*Q}{2} \right)^*Q=\dfrac{1}{2}PD_Q^2=\dfrac{1}{2}D_Q^2P=\dfrac{1}{2}D_QPD_Q \qquad [\text{since } PD_Q=D_QP].
\]
This shows that the space $\mathcal D_Q$ is invariant under $P$ and that $\dfrac{1}{2}P|_{\mathcal D_Q}, \dfrac{1}{2}P|_{\mathcal D_Q}$ satisfy the fundamental equations (as in Theorem \ref{exist-tetra}) for the triple $\left( \dfrac{P+P^*Q}{2}, \dfrac{P+P^*Q}{2}, Q \right)$. Also since $P$ is a contraction, it follows that the numerical radius $\omega \left(\dfrac{1}{2}P|_{\mathcal D_Q}+z\,\dfrac{1}{2}P|_{\mathcal D_Q} \right)\leq 1$ and the condition-(3) of Theorem \ref{thm:E-contraction} follows trivially for $F_1=F_2=\dfrac{1}{2}P|_{\mathcal D_Q}$. Hence by Theorem \ref{thm:E-contraction}, $\left( \dfrac{P+P^*Q}{2}, \dfrac{P+P^*Q}{2}, Q \right)$ is an $\E$-contraction. Applying the same argument to the doubly commuting contractions $P, -Q$ we have that $\left( \dfrac{P-P^*Q}{2}, \dfrac{P-P^*Q}{2}, -Q \right)$ is also an $\E$-contraction.

Now if $Q=Q_1\oplus Q_2$ is the canonical decomposition of $Q$ with respect to $\HS=\HS_1\oplus \HS_2$, then the canonical decomposition of $-Q$ is $-Q=-Q_1 \oplus -Q_2$ with respect to the same orthogonal decomposition of $\HS$, that is, $\HS=\HS_1\oplus \HS_2$. We now consider the canonical decomposition of the $\E$-contraction $\left( \dfrac{P+P^*Q}{2}, \dfrac{P+P^*Q}{2}, Q \right)$ as in Theorem \ref{thm:E-decomp1} and have that $\HS_1,\HS_2$ are reducing subspaces for $\dfrac{P+P^*Q}{2}$. Similarly, considering the canonical decomposition of $\left( \dfrac{P-P^*Q}{2}, \dfrac{P-P^*Q}{2}, -Q \right)$, we have that $\HS_1,\HS_2$ are reducing subspaces for $\dfrac{P-P^*Q}{2}$. Hence $\HS_1, \HS_2$ are reducing subspaces for $\dfrac{P+P^*Q}{2}+\dfrac{P-P^*Q}{2}=P$ and the proof is complete.

\end{proof}

We shall use the following terminologies from here onward.
\begin{defn}
A contraction $T$ acting on a Hilbert space $\HS$ is said to be
\begin{enumerate}
\item[(i)] an \textit{atom} if $T$ is either a unitary or a c.n.u. contraction ;

\item[(ii)] an \textit{atom of type} $A1$ if $T$ is a unitary ; an \textit{atom of type} $A2$ if $T$ is a c.n.u. contraction ;

\item[(iii)] a \textit{non-atom} if $T$ is neither a unitary nor a c.n.u. contraction. 

\end{enumerate}

\end{defn}

Now we present a refinement of Theorem \ref{thm:decomp1} for any finite number of doubly commuting contractions which is one of the main results of this paper. We do not know if this result exists in the literature.

\begin{thm}\label{thm:decomp2}
	
Let $T_1,\dots, T_n$ be doubly commuting contractions acting on a Hilbert space $\HS$. Then there corresponds a decomposition of $\HS$ into an orthogonal sum of $2^n$ subspaces $\HS_1,\dots, \HS_{2^n}$ of $\HS$ such that 
\begin{enumerate}

\item each $\HS_j$, $1\leq j \leq 2^n$, is a joint reducing subspace for $T_1,\dots, T_{n}$;

\item $T_i|_{\HS_{j}}$ is either a unitary or a c.n.u. contraction for all $i=1,\dots , n$ and for $j= 1,\dots , 2^n$;

\item if $\mathcal F_{n}=\{(\mathcal A_{j}\,, \, \mathcal H_{j}): 1\leq j \leq 2^n  \}$, where $\mathcal A_{j}=( T_1|_{\HS_{j}},\dots , T_n|_{\HS_{j}})$ for each $j$ and if $\mathcal M_{n}$ is the set of all functions $f:\{1, \dots , n  \} \rightarrow \{A1\,,\,A2\}$, then corresponding to every $g \in \mathcal M_n$ there exists a unique element $(\mathcal A_g, \HS_g) \in \mathcal F_n$ such that the $k\textsuperscript{th}$ component of $\mathcal A_g$ is of the type $g(k)$, for $1\leq k \leq n$;
 
\item $\HS_1$ is the maximal joint reducing subspace for $T_1, \dots , T_n$ such that $T_1|_{\HS_1}, \dots, T_n|_{\HS_1}$ are unitaries;

\item $\HS_{2^n}$ is the maximal joint reducing subspace for $T_1, \dots , T_n$ such that $T_1|_{\HS_{2^n}}, \dots, T_n|_{\HS_{2^n}}$ are c.n.u. contractions.

\end{enumerate}

One or more members of $\{ \HS_j:\;j=1,\dots ,2^n \}$ may coincide with the trivial subspace $\{0\}$.

\end{thm}

\begin{proof}

We prove by mathematical induction on $n$, the number of doubly commuting contractions. For $n=1$ it is just the canonical decomposition of a contraction and the theorem holds trivially. We shall show here a proof for the $n=2$ case to make our algorithm transparent to the readers. Let $T_1,T_2$ be doubly commuting contractions on $\HS$ and let $T_1=T_{11}\oplus T_{12}$ be the canonical decomposition of $T_1$ with respect to $\HS=\widetilde{\HS}_1 \oplus \widetilde{\HS}_2$, where $T_{11}=T_1|_{\widetilde{\HS}_1}$ is a unitary and $T_{12}=T_1|_{\widetilde{\HS}_2}$ is a c.n.u. contraction. By Theorem \ref{thm:key}, $\widetilde{\HS_1}, \widetilde{\HS_2}$ are reducing subspaces for $T_2$ and suppose $T_2=T_{21}\oplus T_{22}$ with respect to $\HS=\widetilde{\HS}_1 \oplus \widetilde{\HS}_2$. We now consider the doubly commuting operator pairs $(T_{11},T_{21})$ acting on $\widetilde{\HS}_1$ and $(T_{12}, T_{22})$ acting on $\widetilde{\HS}_2$. Note that $T_{11}$ and $T_{12}$ are atoms and thus we shall consider the canonical decomposition of the second components, i.e. $T_{21}$ and $T_{22}$. Let $T_{21}=T_{211}\oplus T_{212}$ be the canonical decomposition of $T_{21}$ with respect to $\widetilde{\HS}_1=\widetilde{\HS}_{11}\oplus \widetilde{\HS}_{12}$, where $T_{211}=T_{21}|_{\widetilde{\HS}_{11}}$ is a unitary and $T_{212}=T_{21}|_{\widetilde{\HS}_{12}}$ is a c.n.u. contraction. Then once again by Theorem \ref{thm:key}, $\widetilde{\HS}_{11},\widetilde{ \HS}_{12}$ are reducing subspaces for $T_{11}$ as it doubly commutes with $T_{21}$. Let $T_{11}=T_{111}\oplus T_{112}$ with respect to $\widetilde{\HS}_1=\widetilde{\HS}_{11}\oplus \widetilde{\HS}_{12}$. Thus, we obtain two new doubly commuting pairs from the pair $(T_{11}, T_{21})$ namely $(T_{111}, T_{211})$ acting on $\widetilde{\HS}_{11}$ and $(T_{112}, T_{212})$ acting on $\widetilde{\HS}_{12}$. Note that the pair $(T_{111}, T_{211})$ acting on $\widetilde{\HS}_{11}$ consists of unitaries only whereas in the pair $(T_{112}, T_{212})$ acting on $\widetilde{\HS}_{12}$, $T_{112}$ is a unitary and $T_{212}$ is a c.n.u. contraction. Similarly, if $T_{22}=T_{221} \oplus T_{222}$ is the canonical decomposition of $T_{22}$ with respect to $\widetilde{\HS}_{2}=\widetilde{\HS}_{21}\oplus \widetilde{\HS}_{22}$, where $T_{221}=T_{22}|_{\widetilde{\HS}_{21}}$ is a unitary and $T_{222}=T_{22}|_{\widetilde{\HS_{22}}}$ is a c.n.u contraction, then by Theorem \ref{thm:key}, ${\widetilde{\HS}_{21}}, {\widetilde{\HS}_{22}}$ reduce $T_{12}$. Let $T_{12}=T_{121}\oplus T_{122}$ with respect to $\widetilde{\HS}_{2}=\widetilde{\HS}_{21}\oplus \widetilde{\HS}_{22}$. Then in the pair $(T_{121},T_{221})$ acting on $\widetilde{\HS}_{21}$, we have that $T_{121}$ is a c.n.u. contraction and $T_{221}$ is a unitary. Also, in the pair $(T_{122}, T_{222})$ acting on $\HS_{22}$, we have that both $T_{122}$ and $T_{222}$ are c.n.u. contractions.
 Thus, if we denote by $\HS_1,\HS_2,\HS_3,\HS_4$ the spaces $\widetilde{\HS}_{11}, \widetilde{\HS}_{12}, \widetilde{\HS}_{21}, \widetilde{\HS}_{22}$ respectively, then we have the following:
 
\begin{enumerate}

\item each $\HS_j$, $1\leq j \leq 4$, is a joint reducing subspace for $T_1,T_{2}$;

\item $T_i|_{\HS_{j}}$ is either a unitary or a c.n.u. contraction for all $i=1,2$ and for $j= 1,\dots , 4$;

\item if $\mathcal F_{2}=\{(\mathcal A_{j}\,, \, \mathcal H_{j}): j =1, \dots, 4  \}$, where $\mathcal A_{j}=( T_1|_{\HS_{j}},T_2|_{\HS_{j}})$ for each $j$ and corresponding to every $g \in \mathcal M_2$ there exists a unique element $(\mathcal A_g, \HS_g) \in \mathcal F_2$ such that the $k\textsuperscript{th}$ component of $\mathcal A_g$ is of the type $g(k)$, for $k=1,2$;
 
\item $\HS_1$ is the maximal joint reducing subspace for $T_1, T_2$ such that $T_1|_{\HS_1},T_2|_{\HS_1}$ are unitaries;

\item $\HS_{4}$ is the maximal joint reducing subspace for $T_1,T_2$ such that $T_1|_{\HS_{4}}, T_2|_{\HS_{2^2}}$ are c.n.u. contractions.

\end{enumerate}
Thus, the theorem holds for $n=2$. Suppose that the theorem holds for $n=s$ and let $T_1,\dots ,T_s$ be doubly commuting contractions acting on $\HS$. Let $T$ be a contraction that doubly commutes with $T_1,\dots, T_s$. Then by Theorem \ref{thm:key}, each of $\HS_1,\dots, \HS_{2^s}$ reduces $T$. Consider the set of operator tuples $$\{(T_1|_{\HS_k}, \dots, T_s|_{\HS_k}, T|_{\HS_k})\,:\, 1\leq k \leq 2^n  \}$$ and choose an arbitrary member from it, say $(T_1|_{\HS_i}, \dots, T_s|_{\HS_i}, T|_{\HS_i})$. Note that $T_1|_{\HS_i}, \dots, T_s|_{\HS_i}$ are all atoms by the induction hypothesis. We perform the canonical decomposition of $T|_{\HS_i}$ and thus $\HS_i$ is orthogonally decomposed into two parts say $\HS_{i1}, \HS_{i2}$ such that $T|_{\HS_{i1}}$ is a unitary and $T|_{\HS_{i2}}$ is a c.n.u. contraction. Once again by Theorem \ref{thm:key}, $\HS_{i1}, \HS_{i2}$ reduce each of $T_1|_{\HS_i}, \dots, T_s|_{\HS_i}$. Thus, the tuple $(T_1|_{\HS_i}, \dots, T_s|_{\HS_i}, T|_{\HS_i})$ orthogonally splits into two parts namely the tuples $A_i=(T_1|_{\HS_{i1}}, \dots, T_s|_{\HS_{i1}}, T|_{\HS_{i1}})$ and $B_i=(T_1|_{\HS_{i2}}, \dots, T_s|_{\HS_{i2}}, T|_{\HS_{i2}})$ with respect to $\HS_i=\HS_{i1}\oplus \HS_{i2}$, where $A_i,B_i$ consist of atoms only. Needless to mention that this algorithm holds for each $i=1,\dots ,2^s$. Thus, the original Hilbert space $\HS$ orthogonally decomposes into the following $2^{s+1}$ parts: $$\HS=(\HS_{11}\oplus\HS_{12})\oplus (\HS_{21}\oplus \HS_{22})\oplus \dots \oplus (\HS_{2^s\,1}\oplus \HS_{2^s\,2}),$$ each of which is a joint reducing subspace for $T_1,\dots ,T_s, T$. Consequently we have a set of $2^{s+1}$ tuples $\{ (T_1|_{\HS_{jk}}, \dots, T_s|_{\HS_{jk}}, T|_{\HS_{jk}}):\, 1 \leq j \leq 2^s \, \& \, k=1,2 \}$ and each of these tuples consists of atoms only. The other parts of the theorem follow naturally from the induction hypothesis, Theorem \ref{thm:key} and the above canonical decomposition of $T$. Hence the proof is complete.

\end{proof}

\subsection{Wold decomposition of doubly commuting isometries}
As a special case of Theorem \ref{thm:decomp2}, we obtain the following Wold decomposition of doubly commuting isometries which is a refinement of Theorem \ref{thm:Wold1} for doubly commuting isometries. Recall that $B1,B2$ stand for a pure isometry and a c.n.i. contractions respectively.

\begin{thm}\label{thm:Wold2}

Let $V_1,\dots, V_n$ be doubly commuting isometries acting on a Hilbert space $\HS$. Then there corresponds an orthogonal decomposition of $\HS$ into $2^n$ subspaces $\HS_1,\dots, \HS_{2^n}$ such that 
\begin{enumerate}
\item each $\HS_j$, $1\leq j \leq 2^n$, is a joint reducing subspace for $V_1,\dots, V_{n}$;

\item $V_i|_{\HS_{j}}$ is either a unitary or a pure isometry for all $i=1,\dots , n$ and for $j= 1,\dots , 2^n$;

\item if $\mathcal F_{n}=\{(\mathcal A_{j}\,, \, \mathcal H_{j}): 1\leq j \leq 2^n  \}$, where $\mathcal A_{j}=( V_1|_{\HS_{j}},\dots , V_n|_{\HS_{j}})$ for each $j$ and if $\mathcal M_{n}$ is the set of all functions $f:\{1, \dots , n  \} \rightarrow \{A1\,,\,B1\}$, then corresponding to every $g \in \mathcal M_n$ there exists a unique element $(\mathcal A_g, \HS_g) \in \mathcal F_n$ such that the $k\textsuperscript{th}$ component of $\mathcal A_g$ is of the type $g(k)$, for $1\leq k \leq n$;
 
\item $\HS_1$ is the maximal joint reducing subspace for $V_1, \dots , V_n$ such that $V_1|_{\HS_1}, \dots, V_n|_{\HS_1}$ are unitaries;

\item $\HS_{2^n}$ is the maximal joint reducing subspace for $V_1, \dots , V_n$ such that $V_1|_{\HS_{2^n}}, \dots, V_n|_{\HS_{2^n}}$ are pure isometries.

\end{enumerate}

One or more members of $\{ \HS_j:\;j=1,\dots ,2^n \}$ may coincide with the trivial subspace $\{0\}$.

\end{thm}

We skip the proof. Indeed, a proof to this theorem uses similar arguments as in the proof of Theorem \ref{thm:decomp2}. Needless to mention that the canonical decomposition of a contraction $P$ is nothing but the Wold decomposition of $P$ when $P$ is an isometry. Also, the completely non-unitary part of $P$ becomes the shift operator when $P$ is an isometry.\\

\vspace{0.2cm}

\section{Canonical decomposition of a finite family of doubly commuting c.n.u. contractions} \label{sec:5}

\vspace{0.2cm}

\noindent We have witnessed an explicit canonical decomposition of doubly commuting contractions in the previous section. In this section, we show that the totally c.n.u. tuple $(T_1|_{\HS_{2^n}}, \dots , T_n|_{\HS_{2^n}})$ as in Theorem \ref{thm:decomp2}, can be further decomposed into $2^n$ orthogonal parts by applying an argument similar to that in the proof of Theorem \ref{thm:decomp2}. This is an analogue of Theorem \ref{thmlv1} due to Levan. We begin with an analogue of Theorem \ref{thmlv1} for a particular class of c.n.u. $\E$-contractions. We achieve our desired result by an application of this theorem.

\begin{thm}[\cite{B-P1}, Theorem 5.1] \label{cd1}
    Let $(A,B, P)$ be a c.n.u $\E$-contraction
    on a Hilbert space $\mathcal{H}$. Let $\mathcal{H}_1, \mathcal H_2$ be as in
    Theorem \ref{thmlv1}. If either $P^*$ commutes with $A,B$ or $\mathcal{H}_1$
    is the maximal invariant subspace for $P$ on which $P$ is isometry, then
    \begin{enumerate}
        \item $\mathcal{H}_1$, $\mathcal{H}_2$ reduce $A,B$;
        \item $\left(A|_{\mathcal{H}_1}, B|_{\mathcal{H}_1},P|_{\mathcal{H}_1} \right)$
        is a pure $\E$-isometry;
        \item $\left(A|_{\mathcal{H}_2},B|_{\mathcal{H}_2},P|_{\mathcal{H}_2} \right)$
        is a c.n.i $\E$-contraction.
    \end{enumerate}
    Anyone of $\mathcal{H}_1,\mathcal{H}_2$ can be equal to the trivial subspace $\{0\}$.
\end{thm}

It is merely mentioned that Theorem \ref{thm:key} plays the central role in determining the refined canonical decomposition of doubly commuting contractions as in Theorem \ref{thm:decomp2}. For our purpose of having further decomposition of doubly commuting c.n.u. contractions into doubly commuting pure isometries and c.n.i. contractions, we shall need an analogue of Theorem \ref{thm:key} in this setting and we present it below.

\begin{thm}\label{thm:key1}
Let $P,Q$ be doubly commuting contractions acting on a Hilbert space $\HS$ and let $Q$ be a c.n.u. contraction. If $Q=Q_1\oplus Q_2$ is the orthogonal decomposition of $Q$ as in Theorem \ref{thmlv1} with respect to the orthogonal decomposition $\HS=\HS_1 \oplus \HS_2$, then $\HS_1,\HS_2$ are reducing subspaces for $P$.
\end{thm}

\begin{proof}

We apply the same argument as in the proof of Theorem \ref{thm:key} to the doubly commuting contractions $P,Q$ and have that $\left( \dfrac{P+P^*Q}{2}, \dfrac{P+P^*Q}{2}, Q \right)$ and $\left( \dfrac{P-P^*Q}{2}, \dfrac{P-P^*Q}{2}, -Q \right)$ are c.n.u. $\E$-contractions. Since $P,Q$ doubly commute, it follows that $ \dfrac{P+P^*Q}{2}$ commutes with $Q^*$ and similarly $\dfrac{P-P^*Q}{2}$ commutes with $-Q^*$. Now if $Q=Q_1\oplus Q_2$ with respect to $\HS=\HS_1\oplus \HS_2$ is the decomposition of $Q$ into a pure isometry $Q_1$ and a c.n.i. contraction $Q_2$ as in Theorem \ref{thmlv1}, then by Theorem \ref{cd1} we have that $\HS_1,\HS_2$ are reducing subspaces for $\dfrac{P+P^*Q}{2}$ and $\dfrac{P-P^*Q}{2}$ which leads to the conclusion that $\HS_1, \HS_2$ are reducing subspaces for $\dfrac{P+P^*Q}{2}+\dfrac{P-P^*Q}{2}=P$. Hence the proof is complete.

\end{proof}

Now we present the main result of this section which provides an explicit orthogonal decomposition of a finitely many doubly commuting c.n.u. contractions. We shall use the following terminologies.

\begin{defn}
A contraction $T$ acting on a Hilbert space $\HS$ is said to be
\begin{enumerate}
\item[(i)] a \textit{fundamental c.n.u. contraction} if either $T$ is a pure isometry or a c.n.i. contraction;

\item[(ii)] a \textit{fundamental c.n.u. contraction} of type $B1$ if $T$ is a pure isometry; a \textit{fundamental c.n.u. contraction} of type $B2$ if $T$ is a c.n.i. contraction;
 
\item[(iii)] a \textit{non-fundamental c.n.u.} contraction, if $T$ is a c.n.u. contraction which is neither a unilateral shift nor a c.n.i. contraction.

\end{enumerate}

\end{defn}

\begin{thm}\label{thm:decomp3}
	
Let $T_1,\dots, T_n$ be doubly commuting c.n.u. contractions acting on a Hilbert space $\HS$. Then there corresponds a decomposition of $\HS$ into an orthogonal sum of $2^n$ subspaces $\HS_1,\dots, \HS_{2^n}$ of $\HS$ such that 
\begin{enumerate}

\item each $\HS_j$, $1\leq j \leq 2^n$, is a joint reducing subspace for $T_1,\dots, T_{n}$;

\item $T_i|_{\HS_{j}}$ is either a pure isometry or a c.n.i. contraction for all $i=1,\dots , n$ and for $j= 1,\dots , 2^n$;

\item if $\mathcal F_{n}=\{(\mathcal A_{j}\,, \, \mathcal H_{j}): 1\leq j \leq 2^n  \}$, where $\mathcal A_{j}=( T_1|_{\HS_{j}},\dots , T_n|_{\HS_{j}})$ for each $j$ and if $\mathcal M_{n}$ is the set of all functions $f:\{1, \dots , n  \} \rightarrow \{B1\,,\,B2\}$, then corresponding to every $g \in \mathcal M_n$ there exists a unique element $(\mathcal A_g, \HS_g) \in \mathcal F_n$ such that the $k\textsuperscript{th}$ component of $\mathcal A_g$ is of the type $g(k)$, for $1\leq k \leq n$;
 
\item $\HS_1$ is the maximal joint reducing subspace for $T_1, \dots , T_n$ such that $T_1|_{\HS_1}, \dots, T_n|_{\HS_1}$ are pure isometries;

\item $\HS_{2^n}$ is the maximal joint reducing subspace for $T_1, \dots , T_n$ such that $T_1|_{\HS_{2^n}}, \dots, T_n|_{\HS_{2^n}}$ are c.n.i. contractions.

\end{enumerate}

One or more members of $\{ \HS_j:\;j=1,\dots ,2^n \}$ may coincide with the trivial subspace $\{0\}$.

\end{thm}

\begin{proof}

We follow the same algorithm as in the proof of Theorem \ref{thm:decomp2} and apply Theorem \ref{thm:key1} repeatedly to reach the desired conclusion.

\end{proof}

\section{Canonical decomposition of an infinite family of doubly commuting contractions } \label{sec:6}

\vspace{0.4cm}

\noindent In this Section, we generalize the ideas of the previous sections and extend the same algorithm to decompose infinitely many doubly commuting contractions. Before presenting the general results for an arbitrary infinite set of doubly commuting contractions, we obtain the results for a countably infinite family so that a reader can understand clearly and intuitively how the same idea is being generalized from the finite to infinite case. It is evident from Theorem \ref{thm:key} that if we have a system of doubly commuting contractions $\mathcal S =(T_{\alpha})_{\alpha \in \Lambda}$ acting on a Hilbert space $\HS$ and if $T_{\beta 1}\oplus T_{\beta 2}$ is the canonical decomposition of some $T_{\beta} \in \mathcal S$ with respect to $\HS=\HS_1 \oplus \HS_2$, then $\HS_1 , \HS_2$ are reducing subspaces for all members of $\mathcal S$. Consequently the entire system $\mathcal S$ is orthogonally decomposed according to the canonical decomposition of $T_{\beta}$. The contraction $T_{\beta}$, with respect to whose canonical decomposition the whole system $\mathcal S$ is orthogonally decomposed, will be called the \textit{``center"} for the decomposition. Note that for an operator tuple $\mathcal S= (T_{\alpha})_{ \alpha \in \Lambda }$ acting on $\HS$, if $\widetilde{\HS} \subseteq \HS$ is a joint invariant subspace for $\mathcal S$, then by $\mathcal S|_{\widetilde{\HS}}$ we shall always mean the tuple $(T_{\alpha}|_{\widetilde{\HS}})_{\alpha \in \Lambda}$. Also, if $\HS=\widetilde{\HS_1} \oplus \widetilde{\HS_2}$ and if $\widetilde{\HS_1}, \widetilde{\HS_2}$ are joint reducing subspaces for $\mathcal S$, then $\mathcal S|_{\widetilde{\HS_1}} \oplus \mathcal S|_{\widetilde{\HS_2}}$ denotes the tuple $(T_{\alpha}|_{\widetilde{\HS_1}} \oplus T_{\alpha}|_{\widetilde{\HS_2}})_{\alpha \in \Lambda}$.  Thus, if $\HS=\oplus_{\beta \in J}\; \HS_{\beta}$ and if $\HS_{\beta}$ is a joint reducing subspace for $\mathcal S$ for each $\beta \in J$, then $\oplus_{\beta \in J}\; \mathcal S|_{\HS_{\beta}}$ denotes the tuple $(\oplus_{\beta \in J}\, T_{\alpha}|_{\HS_{\beta}})_{\alpha \in \Lambda}$.

\subsection{The countably infinite case.}

  Let us consider a countably infinite tuple of doubly commuting contractions $\mathcal{T}=\{ T_{n}:\, n\in \mathbb N \}$, where each $T_{n}$ acts on a Hilbert space $\HS$. For our convenience we shall denote this set by the sequence $\mathcal{T}=(T_{n})_{ n=1} ^{\infty}$, whose terms are all distinct. If there are only finitely many non-atoms in $\mathcal{T}$, then the orthogonal decomposition performed by holding the non-atoms as centers will terminate after a finitely many steps which is similar to the finite case described in the previous sections. Also, the canonical decomposition of an atom is trivial and consequently holding it as the center the orthogonal decomposition of the concerned tuple is also trivial. Thus, without loss of generality we assume that there are infinitely many non-atoms in the tuple $\mathcal{T}$. We shall decompose $\mathcal{T}$ step by step starting with $T_1$ being the center. Suppose $T_{11}\oplus T_{12}$ is the canonical decomposition of $T_1$ with respect to $\HS={\HS}_1 \oplus {\HS}_2$, where $T_{11}=T_1|_{{\HS}_1}$ is a unitary and $T_{12}=T_1|_{{\HS}_2}$ is a c.n.u. contraction. Then by Theorem \ref{thm:key}, ${\HS}_1 , {\HS}_2$ are joint reducing subspaces for $\mathcal{T}$ and consequently $\mathcal{T}$ orthogonally decomposes into two sequences, say $\mathcal{T}_1=\{T_{11}, T_{21}, T_{31}, \dots \}=(T_{n1})_{ n=1} ^{\infty}$ and $\mathcal{T}_2=\{ T_{12}, T_{22}, T_{32}, \dots \} = (T_{n2})_{ n=1} ^{\infty}$, where $T_{n1}=T_n|_{{\HS}_1}$ and $T_{n2}=T_n|_{{\HS}_2}$ for all $n\in\mathbb N$. Note that $T_{11}$ in $\mathcal{T}_1$ and $T_{12}$ in $\mathcal{T}_2$, which are the first components of $\mathcal{T}_1$ and $\mathcal{T}_2$ respectively, are atoms. Let
  \[
   \mathcal F_0=\left\{\left( \mathcal T, \HS \right)  \right\}     \text{ and } \mathcal F_1= \left\{ \left( \mathcal T_1, {\HS}_1 \right), \left( \mathcal T_2, {\HS}_2 \right)  \right\}.
   \]
    Next, we move to consider the $2\textsuperscript{nd}$ components as centers for the $2$ newly obtained operator tuples in $\mathcal F_1$. So, we shall decompose $\mathcal{T}_1$ and $\mathcal{T}_2$ holding $T_{21}$ and $T_{22}$ respectively as the centers and obtain $2^2=4$ new tuples acting on $4$ orthogonal parts of $\HS$, say ${\HS}_{11}, {\HS}_{12}, {\HS}_{21}$ and ${\HS}_{22}$ respectively, where $\HS_1 =\HS_{11}\oplus \HS_{12}$ and $\HS_2=\HS_{21} \oplus \HS_{22}$. Let us denote them by 
\begin{align*}
\mathcal{T}_{11} & = \{T_{111}, T_{211}, T_{311}, \dots  \}=(T_{n11})_{ n=1} ^{\infty} \\
\mathcal{T}_{12} & = \{T_{112}, T_{212}, T_{312}, \dots  \}=(T_{n12})_{ n=1} ^{\infty} \\
\mathcal{T}_{21} & = \{T_{121}, T_{221}, T_{321}, \dots  \}=(T_{n21})_{ n=1} ^{\infty} \\
\mathcal{T}_{22} & = \{T_{122}, T_{222}, T_{322}, \dots  \}=(T_{n22})_{ n=1} ^{\infty}
\end{align*}    
 respectively. Let $$\mathcal F_2 =\{(\mathcal T_{11}, {\HS}_{11}), (\mathcal T_{12}, {\HS}_{12}), (\mathcal T_{21}, {\HS}_{21}), (\mathcal T_{22}, {\HS}_{22})   \}.$$ Note that the first two components of each tuple in $\mathcal F_2$ are atoms. Continuing this process we see that at $n \textsuperscript{th}$ stage we obtain $2^n$ operator tuples (along with the corresponding $2^n$ orthogonal splits of $\HS$) in $\mathcal F_n$ each of whose first $n$ components are atoms. If we denote by $\HS_1^n, \dots, \HS_{2^n}^n$, the orthogonal splits of $\HS$ which are obtained at the $n \textsuperscript{th}$ stage, then $\HS=\oplus_{i=1}^{2^n}\; \HS_i^n$ and this holds for all $n \in \mathbb N$. Let
 \[
 \mathcal F_{\mathbb N}= \cup_{n \in \mathbb N \cup \{ 0 \}}\, \mathcal F_n\,.
 \]
 We define a relation `$\leq $' on $\mathcal F_{\mathbb N}$ in the following way: for any two members $(\mathcal A, \mathcal H_A), (\mathcal B , \mathcal H_B)$ in $\mathcal F_{\mathbb N}$, $(\mathcal A, \mathcal H_{\mathcal A})\leq (\mathcal B , \mathcal H_{\mathcal B})$ if and only if $\mathcal H_{\mathcal B} \subseteq \mathcal H_{\mathcal A}$, $\mathcal H_{\mathcal B}$ is a joint reducing subspace for $\mathcal A$ and $\mathcal B= \mathcal A|_{\mathcal H_{\mathcal B}}$. Evidently $(\mathcal F_{\mathbb N}, \leq)$ is a partially ordered set and that $(\mathcal T, \HS) \leq (\mathcal A , \mathcal H_{\mathcal A})$ for any $(\mathcal A , \mathcal H_{\mathcal A}) \in \mathcal F_{\mathbb N}$.

\begin{defn}
A proper subset $\mathcal G $ of $\mathcal F_{\mathbb N}$ is called a \textit{maximal totally ordered set} if $(\mathcal G, \leq)$ is a totally ordered subset of $(\mathcal F_{\mathbb N}, \leq)$ and for any $(\mathcal A , \mathcal H_{\mathcal A}) \in \mathcal F_{\mathbb N} \setminus \mathcal G$, $\mathcal G \cup \{ (\mathcal A , \mathcal H_{\mathcal A}) \}$ is not a totally ordered subset of $\mathcal F_{\mathbb N}$.
\end{defn}
Note that every maximal totally ordered set contains the initial pair $(\mathcal T, \HS)$. Also, in a maximal totally ordered set, we can arrange and express the elements as an increasing chain. For examples, the following two subsets of $\mathcal F_{\mathbb N}$ are maximal totally ordered sets (with respect to the notations described above) where the elements are arranged in increasing order:
\begin{align*}
& \left\{ \big( (T_n)_{n=1}^{\infty}, \HS \big), \left( (T_{n1})_{n=1}^{\infty}, {\HS}_{1} \right), \left( (T_{n11})_{n=1}^{\infty}, {\HS}_{11} \right), \left( (T_{n111})_{n=1}^{\infty}, {\HS}_{111} \right), \dots   \right \} \\
& \left\{ \big( (T_n)_{n=1}^{\infty}, \HS \big), \left( (T_{n2})_{n=1}^{\infty}, {\HS}_2 \right), \left( (T_{n21})_{n=1}^{\infty}, {\HS}_{21} \right), \left( (T_{n211})_{n=1}^{\infty}, {\HS}_{211} \right), \dots   \right \}
\end{align*}

\begin{rem}
Note that a maximal totally ordered set in $\mathcal F_{\mathbb N}$ contains exactly one element from each $\mathcal F_n$, $( n=0,1,2,\dots )$. 
\end{rem}

\begin{defn}
Two maximal totally ordered subsets $\mathcal G_1$ and $\mathcal G_2$ in $\mathcal F_{\mathbb N}$, whose elements are arranged in increasing order in the following way
\begin{align*}
& \mathcal G_1=\left\{ \big( (T_n)_{n=1}^{\infty}, \HS \big), \left( (T_{\alpha_n 1})_{n=1}^{\infty}, \widetilde{\HS}_{\alpha_1} \right), \left( (T_{\alpha_n2})_{n=1}^{\infty}, \widetilde{\HS}_{\alpha_2} \right), \left( (T_{\alpha_n3})_{n=1}^{\infty}, \widetilde{\HS}_{\alpha_3} \right), \dots   \right \} \\
& \mathcal G_2=\left\{ \big( (T_n)_{n=1}^{\infty}, \HS \big), \left( (T_{\beta_n 1})_{n=1}^{\infty}, \widetilde{\HS}_{\beta_1} \right), \left( (T_{\beta_n2})_{n=1}^{\infty}, \widetilde{\HS}_{\beta_2} \right), \left( (T_{\beta_n3})_{n=1}^{\infty}, \widetilde{\HS}_{\beta_3} \right), \dots   \right \}
\end{align*}
are called \textit{different} if there is $k\in \mathbb N$ such that $\left( (T_{\alpha_n k})_{n=1}^{\infty}, \widetilde{\HS}_{\alpha_k} \right) \neq \left( (T_{\beta_n k})_{n=1}^{\infty}, \widetilde{\HS}_{\beta_k} \right)$.
\end{defn}

Note that if $\mathcal G_1$ and $\mathcal G_2$ differ at $k \textsuperscript{th}$ position, then $\left( (T_{\alpha_n s})_{n=1}^{\infty}, \widetilde{\HS}_{\alpha_s} \right) \neq \left( (T_{\beta_n s})_{n=1}^{\infty}, \widetilde{\HS}_{\beta_s} \right)$ for any $s\geq k$.

\begin{defn}
Let $\mathcal G = \{(\mathcal A_n, \widetilde{\HS}_n)\,:\,n \in \mathbb N \cup \{0 \} \}$ be a maximal totally ordered set such that $(\mathcal A_n, \widetilde{\HS}_n) \leq (\mathcal A_{n+1}, \widetilde{\HS}_{n+1})$ for all $n\in \mathbb N \cup \{ 0 \}$. If $\widetilde{\HS}_{\mathcal G}=\cap_{n=0}^{\infty}\, \widetilde{\HS}_n$ and $\mathcal A_{\mathcal G}= \mathcal A_0|_{\widetilde{\HS}_{\mathcal G}}=\mathcal T|_{\widetilde{\HS}_{\mathcal G}}$, then the pair $(\mathcal A_{\mathcal G} , \widetilde{\HS}_{\mathcal G})$ is called the \textit{maximal element} for $\mathcal G$ or the \textit{limit} of the increasing sequence $(\mathcal A_n , \widetilde{\HS}_n)_{n=1}^{\infty}$ .
\end{defn}

\begin{lem}\label{lem:F1}
The operator tuple in a maximal element consists of atoms only. 
\end{lem}

\begin{proof}

Let $\mathcal J = \{(\mathcal A_n, \widetilde{\HS}_n)\,:\,n \in \mathbb N\cup \{ 0 \} \}$ be a maximal totally ordered set whose elements are arranged in increasing order, i.e. $(\mathcal A_n, \widetilde{\HS}_n) \leq (\mathcal A_{n+1}, \widetilde{\HS}_{n+1})$ for all $n\in \mathbb N \cup \{ 0 \}$ and let $(\mathcal A, \widetilde{\HS})$ be the maximal element for $\mathcal J$. Then $\widetilde{\HS}=\cap_{n=1}^{\infty}\, \widetilde{\HS}_n$ and $\mathcal A = \mathcal T|_{\widetilde{\HS}}$. Note that $(\mathcal A_0, \widetilde{\HS}_0)=(\mathcal T, \HS) \in \mathcal F_0$ and $(\mathcal A_n, \widetilde{\HS}_n) \in \mathcal F_n$ for all $n \in \mathbb N$. The first component of the tuple $\mathcal A_1$ is an atom because it is either equal to $T_{11}=T_1|_{\widetilde{\HS}_1}$ or equal to $T_{12}=T_1|_{\widetilde{\HS}_2}$. Similarly, the first $n$ components of $\mathcal A_n$ are all atoms as $(\mathcal A_n, \widetilde{\HS}_n) \in \mathcal F_n$. Since $\widetilde{\HS} \subseteq \widetilde{\HS}_n$ for all $n\in \mathbb N\cup \{ 0 \}$ and $\mathcal A =\mathcal T|_{\widetilde{\HS}}$, it follows that the first $n$ components of $\mathcal A$ are all atoms and this holds for any $n \in \mathbb N$. Thus the components of $\mathcal A$ are all atoms.

\end{proof}

\begin{lem}\label{lem:F2}
The maximal elements of two different maximal totally ordered sets are different.
\end{lem}

\begin{proof}

If two maximal totally ordered sets $\mathcal G_1$ and $\mathcal G_2$ differ at $k \textsuperscript{th}$ position, then we already know that $\left( (T_{\alpha_n s})_{n=1}^{\infty}, \widetilde{\HS}_{\alpha_s} \right) \neq \left( (T_{\beta_n s})_{n=1}^{\infty}, \widetilde{\HS}_{\beta_s} \right)$ for any $s\geq k$. So, the rest of the proof follows trivially from the definition of maximal element.

\end{proof}

Since the maximal elements for different maximal totally ordered sets are different, we can call the collection of all such maximal elements a set. We denote by $\mathcal F_{\infty}$ the set of all maximal elements for the maximal totally ordered sets in $\mathcal F_{\mathbb N}$. Let $\mathcal F = \mathcal F_{\mathbb N} \cup \mathcal F_{\infty}$. Then we can define $\leq$ on $\mathcal F$ and maximal totally ordered sets in $\mathcal F$ in a similar fashion as we have done for $\mathcal F_{\mathbb N}$. Note that if $\mathcal J$ is a maximal totally ordered set in $\mathcal F$, then for any $(\mathcal B_{\alpha}, \HS_{\mathcal B_{\alpha}}) \in \mathcal J$ we have that $(\mathcal T, \HS) \leq (\mathcal B_{\alpha} , \HS_{\mathcal B_{\alpha}}) \leq (\mathcal B, \HS_{\mathcal B})$, where $(\mathcal B, \HS_{\mathcal B})$ is the maximal element for $\mathcal J$.\\

Recall from Section 2, the atoms of types $A1$ and $A2$. An atom is said to be of type $A1$ if it is a unitary and is of type $A2$ if is a c.n.u. contraction.

\begin{lem}\label{lem:F3}
For every sequence $f:\mathbb N \rightarrow \{ A1\,,\, A2 \}$, there is a unique (maximal) element $(\mathcal A_f, \mathcal H_f )$ in $\mathcal F_{\infty}$ such that for every $n\in \mathbb N$, the $n \textsuperscript{th}$ component of $\mathcal A_f$ is of type $f(n)$.
\end{lem}

\begin{proof}
For such a given sequence $f$, it suffices if we determine the maximal totally ordered set $\mathcal G_f$ whose maximal element is $\mathcal A_f$. Obviously the first entity in $\mathcal G_f$ is $(\mathcal T, \HS)$. As we have mentioned before that a maximal totally ordered set contains exactly one element from each $\mathcal F_n$, we shall choose them step by step and arrange them in increasing order to construct the set $\mathcal G_f$. If $f(1)=A1$, then we choose $(\mathcal T_{1}, \widetilde{\HS_1})$ or else if $f(1)=A2$ then we choose $(\mathcal T_{2}, \widetilde{\HS}_2)$ from $\mathcal F_1$. Now we move to choose an element from $\mathcal F_2$. If $f(1)=A1$, i.e. if $(\mathcal T_{1}, \widetilde{\HS_1}) \in \mathcal G_f$, then we choose $(\mathcal T_{11}, \widetilde{\HS}_{11})$ if $f(2)=A1$ or else we choose $(\mathcal T_{12}, \widetilde{\HS}_{12})$ if $f(2)=A2$. On the other hand if $f(1)=A2$, i.e. if $(\mathcal T_2, \widetilde{\HS}_2)\in \mathcal G_f$, then we choose $(\mathcal T_{21}, \widetilde{\HS}_{21})$ provided that $f(2)=A1$ or else we choose $(\mathcal T_{22}, \widetilde{\HS}_{22})$ if $f(2)=A2$. In a similar fashion we choose one element from $\mathcal F_3$ depending on what elements are already chosen from $\mathcal F_1$ and $\mathcal F_2$. We continue the process to choose exactly one element from each $\mathcal F_n$ and arrange them in increasing order to obtain $\mathcal G_f$. Evidently the set $\mathcal G_f$ is unique and hence its maximal element $(\mathcal A_f, \mathcal H_f)$ is unique. This completes the proof.

\end{proof}

\begin{note}
Let $\mathcal M_{\mathbb N}$ be the set of all sequences $f:\mathbb N \rightarrow \{ A1\,,\, A2\}$. Then the cardinality of the set $\mathcal M_{\mathbb N}$ is equal to $2^{\aleph_0}$, where $\aleph_0$ is the cardinality of the set $\mathbb N$.

\end{note}

\begin{lem}\label{lem:F4}
The cardinality of the set $\mathcal F_{\infty}$ is $2^{\aleph_0}$.
\end{lem}

\begin{proof}
By Lemma \ref{lem:F1}, an element in $\mathcal F_{\infty}$ consists of atoms only. Also, by Lemma \ref{lem:F3}, each sequence $f: \mathbb N \rightarrow \{ A1,A2 \}$ determines a unique element in $\mathcal F_{\infty}$. Since the set of such sequences, i.e. $\mathcal M_{\mathbb N}$ has cardinality $2^{\aleph_0}$, it follows that the cardinality of $\mathcal F_{\infty}$ is $2^{\aleph_0}$.
\end{proof}

\noindent \textbf{Notation.} We express the set $\mathcal F_{\infty}$ as $\mathcal F_{\infty}=\{(\mathcal A_{\alpha}\,, \, \mathcal H_{\alpha}): \alpha \in \mathcal M_{\mathbb N} \}$.

\begin{lem}\label{lem:F5}
Let $\mathcal G_1 ,\mathcal G_2$ be different maximal totally ordered sets in $\mathcal F_{\mathbb N}$ and let $(\mathcal A_{\mathcal G_1}, {\HS}_{\mathcal G_1})$ and $(\mathcal A_{\mathcal G_2}, {\HS}_{\mathcal G_2})$ in $\mathcal F_{\infty}$ be their maximal elements respectively. Then the type of atoms in $\mathcal A_{\mathcal G_1}$ and $\mathcal A_{\mathcal G_2}$ differ by at least one component, i.e. there exists $k\in \mathbb N$ such that the atoms at the $k \textsuperscript{th}$ components of $\mathcal A_{\mathcal G_1}$ and $\mathcal A_{\mathcal G_2}$ are of different types.

\end{lem}

\begin{proof}
Since $\mathcal G_1\,, \mathcal G_2$ are different, they are determined by two different sequences, say $f_1, f_2\, : \mathbb N \rightarrow \{ A1, A2 \}$. Consequently, the type of atoms in $\mathcal A_{\mathcal G_1}$ and $\mathcal A_{\mathcal G_2}$ differ by at least one component.
\end{proof}

\begin{lem}\label{lem:F6}
For the operator tuple $\mathcal T$ on $\HS$, $\HS= \oplus_{\alpha \in \mathcal M}\; \HS_{\alpha}$ and $\mathcal T = \oplus_{\alpha \in \mathcal M} \; \mathcal A_{\alpha}$.
\end{lem}

\begin{proof}

Suppose $\widetilde{\HS}= \oplus_{\alpha \in \mathcal M} \; \HS_{\alpha} \subsetneq \HS$ and let $\widetilde{\HS}^{\perp} = \HS \ominus \widetilde{\HS}$. Let $0 \neq h \in \widetilde{\HS}^{\perp}$. Evidently $P_{\widetilde{\HS}}(h)=0$, where $P_{\widetilde{\HS}}$ is the orthogonal projection of $\HS$ onto $\widetilde{\HS}$. We have that $\HS=\oplus_{i=1}^{2^n}\; \HS_i^n$ for any $n\in \mathbb N$ and thus $h=\oplus_{i=1}^{2^n}\, P_{\HS_i^n}(h)$ for any $n\in \mathbb N$. Let $\mathcal G = \{(\mathcal A_n, \widetilde{\HS}_n)\,:\,n \in \mathbb N \cup \{0 \} \}$ be a maximal totally ordered set and let $(\mathcal A_{\mathcal G} , \widetilde{\HS}_{\mathcal G})$ be its maximal element. Since $P_{\widetilde{\HS_{\mathcal G}}}(h)=0$, it follows that the sequence $\{ P|_{\widetilde{\HS}_n}(h) \}_{n=0}^{\infty}$ converges to $0$. Since this holds for any maximal totally ordered set, we have that the sequence $$\{ h,\;\: P_{\HS^1_1}(h)\oplus P_{\HS^1_2}(h),\;\; P_{\HS^2_1}(h)\oplus P_{\HS^2_2}(h)\oplus P_{\HS^2_3}(h)\oplus P_{\HS^2_4}(h),\;\;\dots \},$$ whose $(n+1) \textsuperscript{th}$ term is $\oplus_{i=1}^{2^n}\, P_{\HS_{i}^n}\, (h)$, converges to $0$. This is a contradiction to the fact that each term of the sequence $\Sigma$ is equal to $h$ and $h \neq 0$. Thus, $\HS= \oplus_{\alpha \in \mathcal M}\; \HS_{\alpha}$ and $\mathcal T = \oplus_{\alpha \in \mathcal M} \; \mathcal A_{\alpha}$.

\end{proof}

\begin{lem}\label{lem:F7}
For every natural number $n$ and for every element $(\mathcal A_{\beta}, \HS_{\beta}) \in \mathcal F_n$, there exist infinitely many maximal totally ordered sets $\mathcal G_{\beta}$ in $\mathcal F_{\mathbb N}$ such that $(\mathcal A_{\beta}, \HS_{\beta})\in \mathcal G_{\beta} $.
\end{lem}

\begin{proof}

It is evident that every maximal totally ordered set is an increasing countably infinite chain which has a limit (the maximal element for the chain) in $\mathcal F_{\infty}$. By Lemma \ref{lem:F4}, we have that the cardinality of set of maximal totally ordered sets with $(\mathcal T, \HS)$ as the initial tuple is $2^{\aleph_0}$. Thus, in a similar fashion we can have infinitely many such chains (maximal totally ordered sets) if we consider $(\mathcal A_{\beta}, \HS_{\beta})$ as its initial element. Needless to mention that any maximal totally ordered set with $(\mathcal A_{\beta}, \HS_{\beta})$ as its initial element is a subset of a maximaly totally ordered set with $(\mathcal T, \HS)$ as its initial element.  Hence, there are infinitely many such $\mathcal G_{\beta}$ in $\mathcal F_{\mathbb N}$ such that $(\mathcal A_{\beta}, \HS_{\beta}) \in \mathcal G_{\beta}$.

\end{proof}

Thus, combining all the results and observations of this section and following the notations defined above, we reach the following generalization of Theorem \ref{thm:decomp2} in the countably infinite setting.

\begin{thm} \label{thm:DECOMP1}

Let $\mathcal T = ( T_n)_{n=1}^{\infty}$ be a family of doubly commuting contractions acting on a Hilbert space $\HS$ which consists of infinitely many non-atoms. If $\mathcal M_{\mathbb N}$ is the set of all sequences $f:\mathbb N \rightarrow \{A1\,,\,A2\}$, then we have the following.

\begin{enumerate}

\item $\HS$ admits an orthogonal decomposition $\HS= \oplus_{\alpha \in \mathcal M_{\mathbb N}} \;  \HS_{\alpha}$ such that each $\HS_{\alpha}$ is a joint reducing subspace for $\mathcal T = (T_n)_{n=1}^{\infty}$.

\item For each $\alpha \in \mathcal M_{\mathbb N}$, the components of operator tuple $(T_n|_{\HS_{\alpha}})_{n=1}^{\infty}$ are all atoms.

\item There is exactly one element say $(\mathcal A_1, \HS_1)$ in $\mathcal F_{\infty}$ such that the components of $\mathcal A_1$ are all atoms of type $A1$ and this is determined by the constant sequence $f_1:\mathbb N \rightarrow \{ A1,A2 \}$, defined by $f_1(n)=A1$ for all $n\in \mathbb N$.

\item There is exactly one element say $(\mathcal A_{2^{\aleph_0}}, \HS_{2^{\aleph_0}})$ in $\mathcal F_{\infty}$ such that the components of $\mathcal A_{2^{\aleph_0}}$ are all atoms of type $A2$ and this is determined by the constant sequence $f_{2^{\aleph_0}}:\mathbb N \rightarrow \{ A1,A2 \}$ defined by $f_{2^{\aleph_0}}(n)=A2$ for all $n\in \mathbb N$.

\item The cardinality of the set $\mathcal F_{\infty} = \{ (\mathcal A_{\alpha}, \HS_{\alpha}) : \alpha \in \mathcal M_{\mathbb N} \}$ is $2^{\aleph_0}$, where $\mathcal A_{\alpha}=(T_n|_{\HS_{\alpha}})_{n=1}^{\infty}$ for each $\alpha \in \mathcal M_{\mathbb N}$ and $\aleph_0$ is the cardinality of the set $\mathbb N$.

\end{enumerate}

One or more members of $\{ \HS_{\alpha}\,:\; \alpha \in \mathcal M_{\mathbb N}\}$ may coincide with the trivial subspace $\{0\}$.

\end{thm}

\vspace{0.3cm}

It is obvious that if we follow the same algorithm (as described above), we can easily obtain an analogue of Theorem \ref{thm:decomp3} for any countable family of doubly commuting c.n.u. contractions which by repeated application of Theorem \ref{thm:key1} must orthogonally split into unilateral shifts and c.n.i. contractions. We do not state this result here. Rather, we would like to present in Section \ref{sec:7}, a more general version of this result which holds for any infinite family of doubly commuting c.n.u. contractions. Before that let us proceed to the next subsection where we shall establish a general version of Theorem \ref{thm:DECOMP1} and this will be one of the main results of this article.

\vspace{0.5cm}

\subsection{The general case.}

In the canonical decomposition of $\mathcal T = (T_n)_{n=1}^{\infty}$ described above, we chose the positions of the centers in an increasing order. In the first stage, we considered the first component $T_1$ as the center to perform the decomposition and obtained the tuples $\mathcal T_1$ and $\mathcal T_2$. In the second stage, we chose the second components of the tuples $\mathcal T_1, \mathcal T_2$ to be the centers and obtained $\mathcal T_{11}, \mathcal T_{12}$ and $\mathcal T_{21},\mathcal T_{22}$ respectively. In the third stage, we fixed the third components of $\mathcal T_{11}, \mathcal T_{12},\mathcal T_{21},\mathcal T_{22}$ as centers. Thus, we followed the natural order $1 \rightarrow 2 \rightarrow 3 \rightarrow \cdots$ and finally achieve the limiting set $\mathcal F_{\infty}$. If we follow a different order, then also we reach the same set $\mathcal F_{\infty}$. We have a clear picture if we deal with a finite tuple $(T_1, \dots , T_n)$, because, then the limiting set will be equal to $\mathcal F_n$  irrespective of the orders $1 \rightarrow 2 \rightarrow 3 \rightarrow \cdots \rightarrow n$ or $ \tau(1)\rightarrow \tau(2)\rightarrow \tau(3) \rightarrow \cdots \rightarrow \tau(n)$ for any bijection $\tau: \{  1,\dots ,n \} \rightarrow \{ 1, \dots ,n \}$. Now, if $\sigma: \mathbb N \rightarrow \mathbb N$ is a bijection and if we choose the order $ \sigma (1)\rightarrow \sigma (2)\rightarrow \sigma (3) \rightarrow \cdots $ in stead of $1 \rightarrow 2 \rightarrow 3 \rightarrow \cdots$, we see that at the first stage we split $T_{\sigma(1)}$ instead of splitting $T_1$ and the whole tuple $\mathcal T$ splits accordingly into two parts. In this case if $\sigma(k)=1$, then $T_1$ will split into atoms at the $k\textsuperscript{th}$ stage of decomposition but the splits of $T_1$ will remain same after $k\textsuperscript{th}$ stage of decomposition if we follow the order $1 \rightarrow 2 \rightarrow 3 \rightarrow \cdots$ too. It is merely said that the same holds for every $T_i$. Thus, rearranging the terms of $\mathcal T$ (according to the bijection $\sigma$) we obtain the ordered tuple $(T_{\sigma(n)})_{n=1}^{\infty}$  but the limiting set $\mathcal F_{\infty}$ remains unchanged. We capitalize this fundamental clue when we deal with the general case.\\

Let us consider a set of doubly commuting contractions $\mathcal{U}=\{T_{\lambda}\,:\, \lambda \in \Lambda  \}$, where $\Lambda$ is any infinite set and each $T_{\lambda}$ acts on a Hilbert space $\HS$. In a similar fashion as in the previous subsection, we assume that there are infinitely many non-atoms in $\mathcal U$ to avoid the triviality. We choose a $\lambda \in \Lambda$ and consider $T_{\lambda}$ as the center. If $T_{\lambda 1 \lambda}\oplus T_{\lambda 2 \lambda}$ is the canonical decomposition of $T_{\lambda}$ with respect to $\HS=\HS_{\lambda 1} \oplus \HS_{\lambda 2}$, then $\HS_{\lambda 1}, \HS_{\lambda 2}$ are joint-reducing subspaces for $\mathcal{U}$ and thus $\mathcal U$ is orthogonally decomposed into $2$ tuples say $\mathcal U_{\lambda 1}=\{T_{{\lambda 1}\, \eta}\,:\, \eta \in \Lambda \},\, \mathcal U_{\lambda 2}=\{T_{{\lambda 2}\, \eta}\,:\, \eta \in \Lambda \}$, whose $\lambda \textsuperscript{th}$ entities ($T_{\lambda 1 \lambda}$ and $T_{\lambda 2 \lambda}$ respectively) are atoms. Next we apply the axiom of choice to have some $\beta \in \Lambda \setminus \{ \lambda \}$ and consider the $\beta \textsuperscript{th}$ components of the tuples $\mathcal U_{\lambda 1}, \mathcal U_{\lambda 2}$, i.e. $T_{\lambda 1 \beta}$ and $T_{\lambda 2 \beta}$ as the centers to perform the canonical decomposition of $\mathcal U_{\lambda 1}, \mathcal U_{\lambda 2}$ respectively. We continue the process and the limiting set $\mathcal F_{\infty}$ must have cardinality $2^{|\Lambda|}$, where $|\Lambda |$ is the cardinality of the set $\Lambda$. This is because here we consider the set of all functions $f:\Lambda \rightarrow \{ A1,A2\}$ and its cardinality is $2^{|\Lambda |}$. If we define $`` \leq "$ in a similar fashion (as in the previous subsection), we can establish analogues of all lemmas proved in the previous subsection in an analogous manner.\\

 Thus we reach the following analogue of the decomposition Theorems \ref{thm:decomp2} $\&$ \ref{thm:DECOMP1} in an arbitrary infinite setting, which are two main results of this paper.

\begin{thm} \label{thm:DECOMP11}

Let $\mathcal U = \{ T_{\lambda}: \lambda \in \Lambda \}$ be an infinite set of doubly commuting contractions acting on a Hilbert space $\HS$ and let there be infinitely many non-atoms in $\mathcal U$. If $\mathcal F_{\infty}=\{(\mathcal A_{\alpha}\,, \, \mathcal H_{\alpha}): \alpha \in \mathcal M_{\Lambda}  \}$, where $\mathcal A_{\alpha}=\{ T_{\lambda}|_{\HS_{\alpha}}\,: \, \lambda \in \Lambda \}$ for each $\alpha \in \mathcal M_{\Lambda}$ and $\mathcal M_{\Lambda}$ is the set of all functions $f:\Lambda \rightarrow \{A1\,,\,A2\}$, then we have the following.

\begin{enumerate}

\item $\HS$ admits an orthogonal decomposition $\HS= \oplus_{\alpha \in \mathcal M_{\Lambda}} \;  \HS_{\alpha}$ such that each $\HS_{\alpha}$ is a joint reducing subspace for $\mathcal U$.

\item For each $\alpha \in \mathcal M_{\Lambda}$, the components of operator tuple $\{ T_{\lambda}|_{\HS_{\alpha}}: \lambda \in \Lambda \}$ are all atoms.

\item There is exactly one element say $(\mathcal A_1, \HS_1)$ in $\mathcal F_{\infty}$ such that the components of $\mathcal A_1$ are all atoms of type $A1$ and this is determined by the constant function $f_1:\Lambda \rightarrow \{ A1,A2 \}$, defined by $f_1(\lambda)=A1$ for all $\lambda\in \Lambda$.

\item There is exactly one element say $(\mathcal A_{2^{|\Lambda|}}, \HS_{2^{|\Lambda|}})$ in $\mathcal F_{\infty}$ such that the components of $\mathcal A_{2^{|\Lambda|}}$ are all atoms of type $A2$ and this is determined by the constant function $f_{2^{\Lambda}}:\Lambda \rightarrow \{ A1,A2 \}$, defined by $f_{2^{\Lambda}}(\lambda)=A2$ for all $\lambda\in \Lambda$.

\item The cardinality of the set $\mathcal F_{\infty} = \{ (\mathcal A_{\alpha}, \HS_{\alpha}) : \alpha \in \mathcal M_{\Lambda} \}$ is $2^{|\Lambda|}$,  where $|\Lambda|$ is the cardinality of the set $\Lambda$.

\end{enumerate}

One or more members of $\{ \HS_{\alpha}\,:\; \alpha \in \mathcal M_{\Lambda} \}$ may coincide with the trivial subspace $\{0\}$.

\end{thm}

As a special case of Theorem \ref{thm:DECOMP11}, we have the following analogue of Theorem \ref{thm:Wold2}, an Wold decomposition for an arbitrary infinite set of doubly commuting isometries.

\begin{thm} \label{thm:WOLD11}

Let $\mathcal V = \{ V_{\eta}: \eta \in I \}$ be an infinite set of doubly commuting isometries acting on a Hilbert space $\HS$ and let there be infinitely many isometries in $\mathcal V$ which are neither unitaries nor pure isometries. If $\mathcal F_{\infty}=\{(\mathcal A_{\gamma}\,, \, \mathcal H_{\gamma}): \gamma \in \mathcal M_I  \}$, where $\mathcal A_{\gamma}=\{ T_{\eta}|_{\HS_{\gamma}}\,: \, \eta \in I \}$ for each $\gamma \in \mathcal M_I$ and $\mathcal M_I$ is the set of all functions $h:I \rightarrow \{A1\,,\,B1\}$, then we have the following.

\begin{enumerate}

\item $\HS$ admits an orthogonal decomposition $\HS= \oplus_{\gamma \in \mathcal M_I} \;  \HS_{\gamma}$ such that each $\HS_{\gamma}$ is a joint reducing subspace for $\mathcal V$.

\item For each $\gamma \in \mathcal M_I$, the operator tuple $\{ T_{\eta}|_{\HS_{\gamma}}: \eta \in I \}$ consist of unitaries and pure isometries.

\item There is exactly one element say $(\mathcal A_1, \HS_1)$ in $\mathcal F_{\infty}$ such that the components of $\mathcal A_1$ are all unitaries and this is determined by the constant function $f_1:I \rightarrow \{ A1,B1 \}$ defined by $f_1(\eta)=A1$ for all $\eta \in I$.

\item There is exactly one element say $(\mathcal A_{2^{|I|}}, \HS_{2^{|I|}})$ in $\mathcal F_{\infty}$ such that the components of $\mathcal A_{2^{|I|}}$ are all pure isometries and this is determined by the constant function $f_{2^{|I|}}:I \rightarrow \{ A1,B1 \}$ defined by $f_{2^{|I|}}(\eta)=B1$ for all $\eta\in I$.

\item The cardinality of the set $\mathcal F_{\infty} = \{ (\mathcal A_{\gamma}, \HS_{\gamma}) : \gamma \in \mathcal M_I \}$ is $2^{|I|}$, where $|I|$ is the cardinality of the set $I$.

\end{enumerate}

One or more members of $\{ \HS_{\gamma}\,:\; \gamma \in \mathcal M_I \}$ may coincide with the trivial subspace $\{0\}$.

\end{thm}

\vspace{0.5cm}

\section{Decomposition of infinitely many doubly commuting c.n.u. contractions} \label{sec:7}

\vspace{0.4cm}

\noindent If we follow the same algorithm as described above, we can easily obtain a canonical decomposition of an infinite family of doubly commuting c.n.u. contractions, whence every member will orthogonally decomposed into pure isometries and c.n.i. contractions. Needless to say that this is an extended version of the Theorem \ref{thm:decomp3} and an analogue of Theorem \ref{thm:DECOMP1} for any infinite set of doubly commuting c.n.u. contractions $\mathcal W =\{ T_{\mu}:\mu \in J \}$. Once again we shall assume without loss of generality that there are infinitely many non-fundamental c.n.u. contractions (i.e. pure isometries or c.n.i. contractions) in $\mathcal W$. It is merely mentioned that each split of $\mathcal W$ will consist of fundamental c.n.u. contractions only. Here we shall denote the limiting set by $\mathcal G_{\infty}$ (which we had denoted by $\mathcal F_{\infty}$ in case of doubly commuting contractions).

\begin{thm} \label{thm:DECOMP12}

Let $\mathcal W =\{ T_{\mu}:\mu \in J \}$ be an infinite set of doubly commuting c.n.u. contractions acting on a Hilbert space $\HS$ and let there be infinitely many non-fundamental c.n.u. contractions in $\mathcal W$. If $\mathcal G_{\infty}=\{(\mathcal A_{\beta}\,, \, \mathcal H_{\beta}): \beta \in \mathcal M_J  \}$, where $\mathcal A_{\beta}= \mathcal W|_{\mathcal H_{\beta}}$ and $\mathcal M_J$ is the set of all functions $g:J \rightarrow \{B1\,,\,B2\}$, then we have the following. 

\begin{enumerate}

\item $\HS$ admits an orthogonal decomposition $\HS= \oplus_{\beta \in \mathcal M_J} \;  \HS_{\beta}$ such that each $\HS_{\beta}$ is a joint reducing subspace for $\mathcal W$.

\item For each $\beta\in \mathcal M_J$, the components of $\{ T_{\mu}|_{\HS_{\beta}}\,:\, \mu \in J \}$ are all fundamental c.n.u. contractions.

\item There is exactly one element say $(\mathcal A_1, \HS_1)$ in $\mathcal G_{\infty}$ such that the components of $\mathcal A_1$ are all fundamental c.n.u. contractions of type $B1$ and this is determined by the constant function $g_1:J \rightarrow \{ B1,B2 \}$, defined by $g(\mu)=B1$ for all $\mu\in J$.

\item There is exactly one element say $(\mathcal A_{2^{|J|}}, \HS_{2^{|J|}})$ in $\mathcal G_{\infty}$ such that the components of $\mathcal A_{2^{|J|}}$ are all fundamental c.n.u. contractions of type $B2$ and this is determined by the constant function $g_{2^{|J|}}:\mathbb N \rightarrow \{ B1,B2 \}$, defined by $g_{2^{|J|}}(\mu)=B2$ for all $\mu \in J$.

\item The cardinality of the set $\mathcal G_{\infty} $ is $2^{|J|}$, where $|J|$ is the cardinality of the set $J$.

\end{enumerate}

One or more members of $\{ \HS_{\beta}\,:\; \beta \in \mathcal M_J \}$ may coincide with the trivial subspace $\{0\}$.

\end{thm}

\end{document}